\documentclass[12pt]{amsart}

\usepackage[T1]{fontenc}
\usepackage[english]{babel}
\usepackage{amsmath, amsthm, amsfonts, mathrsfs, amssymb, float}
\usepackage{mathtools}
\mathtoolsset{centercolon}
\usepackage{booktabs}
\usepackage{amsmath, amsfonts,enumerate}
\usepackage[colorlinks]{hyperref}
\usepackage{tikz}
\usetikzlibrary{arrows}
\usepackage{pifont}
\usepackage{color}
\usepackage{graphicx}

\mathtoolsset{showonlyrefs,showmanualtags}
\usepackage[msc-links]{amsrefs}

\makeatletter \@mparswitchfalse \makeatother

\usepackage{geometry}
\usepackage{enumerate}

\newtheorem{thmx}{Theorem}

\newtheorem*{remark*}{Remark B}

\newtheorem{theorem}{Theorem}
\newtheorem{lemma}[theorem]{Lemma}

\newtheorem{corollary}[theorem]{Corollary}
\newtheorem{proposition}[theorem]{Proposition}

\theoremstyle{definition}
\newtheorem{definition}[theorem]{Definition}

\newcommand{\normmm}[1]{{\left\vert\kern-0.25ex\left\vert\kern-0.25ex\left\vert #1
   \right\vert\kern-0.25ex\right\vert\kern-0.25ex\right\vert}}

\numberwithin{equation}{section}
\numberwithin{theorem}{section}

\let\<\langle
\def\ch{\raise 0.5ex \hbox{$\chi$}}

\let\\\cr
\let\phi\varphi

\let\epsilon\varepsilon

\makeatletter
\@namedef{subjclassname@2020}{\textup{2020} Mathematics Subject Classification}
\makeatother

\begin{document}

\title[Sharp weighted estimates for square functions]{Sharp weighted norm estimates for martingale square functions}

\author[W. Chen]{Wei Chen}
\address{School of Mathematics, Yangzhou University, Yangzhou 225002, China}
 \email{weichen@yzu.edu.cn}

\author[Y. Jiao]{Yong Jiao}
\address{School of Mathematics and Statistics, Central South University, HNP-LAMA, Changsha 410083, China}
\email{jiaoyong@csu.edu.cn}

\author[X. Quan]{Xingyan Quan}
\address{School of Mathematics and Statistics, Central South University, HNP-LAMA, Changsha 410083, China}
\email {quanxingyan@csu.edu.cn}

\author[L. Wu]{Lian Wu}
\address{School of Mathematics and Statistics, Central South University, HNP-LAMA, Changsha 410083, China}
\email{wulian@csu.edu.cn}


 \thanks{This paper was supported by the National Key R$\&$D Program of China (No.2023YFA1010800), the NSFC (Nos.11971419, 12125109, W2411005, 12361131578),  the Provincial Natural Science Foundation of Hunan (Nos.2024JJ1010, 2025ZYJ002, 2024RC1016).}


\begin{abstract}
This paper is devoted to the study of quantitative weighted norm estimates for martingale square functions in both scalar-weighted and matrix-weighted settings. In particular, we introduce the martingale square functions $S_W$ via matrix weights $W$, and then use the matrix $A_p$ condition, introduced in our previous work \cite{ChenQuanJiaoWu}, to characterize the $L_p$ estimate for $S_W$. Our proof mainly relies on the idea of sparse dominations, which leads to the explicit information on the characteristic of the matrix weight involved. For the range  $1<p\leq 2$, our result is sharp in terms of the characteristic of the matrix weight. With some modification on the arguments, we can further improve the result in scalar settings by obtaining the optimal exponent of the characteristic of the weight involved for all indices $1<p<\infty$, addressing a fundamental problem from the classical martingale theory.
 \end{abstract}

\maketitle

\section{Introduction}
Square functions and inequalities associated with them play a significantly important role in the development of probability, harmonic analysis,  ergodic theory and many other areas of mathematics.
Inequalities for martingale square functions, referred to as Burkholder-Gundy inequalities, lie at the heart of these developments.
Let $(\Omega, \mathcal F, \mathbb P)$ be a probability space equipped with a filtration $(\mathcal F_n)_{n\geq 1}$ whose union generates $\mathcal F$. For each $k\geq 1$, denote by $\mathbb E_k$ the conditioned expectation with respect to $\mathcal F_k$. Given  a martingale $f=(f_n)_{n\geq 1}$  adapted to  $(\mathcal F_n)_{n\geq 1}$, the associated difference sequence $df=(d_nf)_{n\geq 1}$ is given by  $d_1f=f_1$ and $d_n f=f_n-f_{n-1}$ for $n\geq 2$, and the square function  of $f$ is defined as $S(f)=(\sum_{k\geq 1}|d_k f|^2)^{1/2}$. The well known inequality, developed by Burkholder and Gundy  in the 1960s,  asserts that for $1<p<\infty$ there is a constant $c_{p}$ depending only on $p$ such that
\begin{equation}\label{classical-BG}
\|S(f)\|_{L_p(\Omega)}\leq c_{p} \|f\|_{L_p(\Omega)}
\end{equation}
For $p=1$, Burkholder \cite{MR208647} (see also Os\k ekowski \cite{MR2853676}) established a weak type $(1,1)$ version of \eqref{classical-BG}.
These profound results provide a probabilistic analogue of the Littlewood-Paley theory. They serve as  crucial tools in many different areas of mathematics. There are numerous generalizations of \eqref{classical-BG}; see for instance Pisier \cite{MR394135} for the extension of \eqref{classical-BG} to vector-valued martingales for which the deep connection between vector-valued martingale inequalities and geometric properties of Banach spaces was found, and Johnson-Schechtman \cite{MR995572} for a far reaching formulation of \eqref{classical-BG} in the context of rearrangement invariant spaces. For further details on martingale inequalities associated with square functions, see the monograph of Weisz \cite{MR1320508}.

The theory of weights has draw much attention in recent years. It all begun with the introduction of $A_p$ weights and the establishment of the weighted Hardy-Littlewood  inequalities by Muckenhoupt  \cite{MR293384}.  Izumisawa and Kazamaki \cite{MR436313} extended the $A_p$ condition to the probabilistic context. Under this condition and the extra reverse H\"{o}lder assumption on weights, they  established the weighted Doob maximal inequality. The extra reverse H\"{o}lder assumption was later successfully removed by Jawerth \cite{MR833361}. Let us briefly recall this result. By a weight on the probability space $(\Omega, \mathcal F, \mathbb P)$  we mean a positive function belonging to $L_1(\Omega)$. Given $1<p<\infty$,
it was shown in \cite{MR833361} that the Doob maximal operator is bounded on $L_p^w(\Omega)$ if and only if $w\in A_p$, i.e., the $A_p$ characteristic
$$[w]_{A_p}=\sup_{n\geq 1}\|\mathbb E_n (w) \mathbb E_n(w^{\frac{1}{1-p}})^{p-1} \|_{L_\infty(\Omega)}$$
is finite. This $A_p$ condition is useful in describing weighted norm estimates of other classical operators as well, see \cites{MR3748572,MR3767376,MR3916937,MR4926944} and extensive bibliographies therein.
In \cite{MR544802}, Bonami and L\'{e}pingle considered the weighted version of \eqref{classical-BG}. Given $1<p<\infty$ and $w\in \widehat{A}_\infty(\Omega)\cap \mathbb S$ (c.f. \cite{MR544802} or \cite[Definitions 2.6, 5.17]{MR3895312} for the introduction of these two conditions), their result can be reformulated as follows:
\begin{equation}\label{weighted-BG}
\|S(f)\|_{L_p^w(\Omega)}\leq c_{p,w} \|f\|_{L_p^w(\Omega)}
\end{equation}
where $c_{p,w}$ is a constant depending only on $p$ and $w$.
It is natural to ask whether the $A_p$ condition is sufficient for the  characterization of the above estimate. One may turn to the recent work \cite{MR4926944} of Domelevo, Petermichl and \v{S}kreb for an affirmative answer.

In addition, there is a very interesting aspect of the weight theory, concerning the extraction of the optimal dependence of the constants involved on the characteristic of a weight. This topic firstly appeared in Buckley's work \cite{MR1124164}, where the best exponent of the characteristic of the weight involved in weighted Hardy-Littlewood maximal inequalities was determined. A similar result on weighted Doob  inequalities can be found in Tanaka and Terasawa \cite[Corollary 4.5]{MR3004953}. Hyt\"{o}nen \cite{MR2912709}  addressed the famous $A_2$ conjecture for general Calder\'on-Zygmund operators. See also \cite{MR3748572,MR2652182,MR2524658} and references therein for similar results regarding other operators.

Our main concern of this paper is a sharp version of \eqref{weighted-BG} under the $A_p$ condition. More precisely, writing $c_{p,w}$ as $c_p[w]_{A_p}^{\kappa_p}$ in \eqref{weighted-BG}, we are interested in  the following fundamental problem:

\vspace{+0.25cm}

\noindent\textsf{\centerline{Given $\textsf w\in A_p$, what is the optimal choice $\kappa_p$ such that \eqref{weighted-BG} holds true~?}}

\vspace{+0.25cm}

\noindent Surprisingly, though the study of weighted inequalities has a rich history, it seems that the above problem has never really been satisfactorily resolved! There are substantial discussion on this problem in some special contexts. For the dyadic square function defined via the Haar system, Buckley \cite{MR1124164}  provided a quantitative weighted $L_2$ bound with $\kappa_2=3/2$. Later, Hukovic \cite{MR2697378} and Hukovic-Treil-Volberg \cite{MR1771755}  improved Buckley's result by showing that the optimal choice of $\kappa_2$ is actually $1$ rather than $3/2$. See also Wittwer \cite{MR1748283} and Petermichl and Pott \cite{MR1873024} for a different proof. In \cite{MR2854179}, Cruz-Uribe, Martell and P\'{e}rez received a considerable progress by establishing \eqref{weighted-BG} for dyadic square functions with $\kappa_p=\max\{1/2,(p-1)^{-1}\}$; moreover, they proved that this exponent is the best possible. For continuous-time martingales  with continuous-path, Ba\~nuelos and Os\k ekowski \cite{MR3748572} obtained \eqref{weighted-BG} for $p=2$ by using the so-called Bellman function method. Their result is sharp in terms of the dependence on $[w]_{A_2}$. But, the additional continuous-path hypothesis prevents us  deducing the corresponding result for discrete-time martingales. We should also mention the paper \cite{MR3985127}, containing the discussion of \eqref{weighted-BG} for $p=2$ in the special case that the filtration is atomic. Recently,  Domelevo, Petermichl and \v{S}kreb \cite{MR4926944} obtained a sharp weighted $L_p$ estimate for the maximal operator $Y^*$ of $Y$ with respect to $X$, where $Y$ and $X$ are uniformly integrable c\`{a}dl\`{a}g Hilbert space valued martingales and $Y$ is differentially subordinate to $X$ via the square bracket process. As a direct consequence of \cite[Corollary 1]{MR4926944}, one can infer \eqref{weighted-BG} with $\kappa_p=\max\{1,(p-1)^{-1}\}$ which, however, is only optimal when $1<p\leq 2$.

The present paper solves the problem discussed above in full generality. One of our principal results is the following.

\begin{thmx}\label{main1}
Let $1<p<\infty$. Given a scalar weight $w\in A_p$, we have
$$\|S f\|_{L_p^w(\Omega)}\lesssim_{p}[w]_{A_p}^{\max\{\frac12 ,\frac{1}{p-1}\}}\|f\|_{L_p^w(\Omega)},\quad \forall f\in L_p^w(\Omega).$$
The estimate is sharp in terms of the dependence on $[w]_{A_p}$.
\end{thmx}

Actually, we will study the above statement from a much wider matrix-weighted perspective. This is a newly emerging area, although the investigation can be tracked back to Wiener and Masani  \cite{MR97859}, where the matrix-weighted  $L_2$ space  was introduced and applied on the prediction theory for multivariate stochastic processes. In 1997, Treil and Volberg \cite{MR1428818} generalized the $A_2$ condition to the matrix-weighted setting and  used this condition to describe the matrix-weighted  $L_2$ boundedness of Hilbert transforms. By using the language of metrics and their averagings, Volberg \cite{MR1423034} introduced the martix $A_p$ condition  and proved that the Hilbert transform is bounded on the matrix-weighted $L_p$ space if and only if the matrix weight belongs to the matrix $A_p$ class. See also Nazarov and Treil \cite{MR1428988} for the introduction of matrix $A_p$ weights. These works pave the way towards the consideration of matrix-weighted inequalities, after which lots of papers have been devoted to this area, see for instance, \cite{MR3452715, 2210.09443, MR1813604,2402.06961, MR2015733}.

In the recent article \cite{ChenQuanJiaoWu}, the authors extended the concept of matrix $A_p$ condition to the probabilistic context.

The main contribution of the paper \cite{ChenQuanJiaoWu} is twofold: the first is to construct  reducing matrix-valued functions in the probabilistic context so that the matrix $A_p$ condition can be properly formulated and the second is to show that under this condition  the matrix-weighted Doob maximal inequalities hold true for vector-valued martingales.

Based on \cite{ChenQuanJiaoWu}, we consider in this paper the matrix-weighted estimates for martingale square functions. To state the result clearly, we fix some notations. Let $1\leq d<\infty$. A matrix weight $W:\Omega\rightarrow \mathbb M_d(\mathbb C)$  is a $d\times d$  self-adjoint matrix function on  $\Omega$ (with integrable entries) such that $W(x)$ is positive definite for almost all $x\in \Omega$. Given a matrix weight $W$ and $1\leq p<\infty$, we define the matrix-weighted square function by
$$S_Wf=\Big(\sum_{k\geq 1}\| W^{\frac1p}d_{k}(W^{-\frac1p}f)\|_{\mathbb C^d}^2\Big)^{1/2},$$
where $f=(f_n)_{n\geq1}$ are $\mathbb C^d$-valued martingales. Note that in sharp contrast with the scalar-weighted setting, one needs to define
the square functions $S_W$  via matrix weights $W$. This  phenomenon
was already uncovered by Christ and Goldberg \cite{MR1813604} and Goldberg \cite{MR2015733}  while establishing
matrix-weighted versions of Hardy-Littlewood inequalities. Our second main result of this paper, concerning the extension of \eqref{weighted-BG} to the matrix-weighted setting, reads as follows.

\begin{thmx}\label{main2}
Let $1<p<\infty$. Given a matrix weight $W\in A_p$, we have $$\|S_Wf\|_{L_p(\Omega)}\lesssim_{p,d}[W]_{A_p}^{\max\{\frac12+\frac{1}{p(p-1)},\frac{1}{p-1}\}}\|f\|_{L_p(\Omega;\mathbb C^d)},\quad \forall f\in L_p(\Omega;\mathbb C^d).$$
For $1<p\leq 2$, the exponent of $[W]_{A_p}$ equals to $1/(p-1)$ which is sharp.
\end{thmx}

This result has been investigated for dyadic square functions by some authors. Hyt\"{o}nen, Petermichl and Volberg \cite{MR3936542} addressed the matrix $A_2$ conjecture for the dyadic square function. Isralowitz \cite{MR4159390} refined the sparse domination from \cite{MR3936542} and then utilized it to prove sharp matrix-weighted type inequalities for dyadic square functions in the range $1<p\leq2$. Recently, Treil \cite{MR4591773} generalized the result of Hyt\"{o}nen, Petermichl and Volberg  to the non-homogeneous situation.
Note that, for $1<p\leq 2$, the exponent of the characteristic of the weights in Theorem \ref{main2} is the same with that of the dyadic setting, and thus is  sharp. Also, note that for $2<p<\infty$,   the exponent is most likely not optimal and we leave it as a problem to the interested reader

Our approach to Theorem \ref{main2} depends on the idea of sparse domination. As a powerful tool, sparse domination nowadays has been widely used  in analytic settings, especially in the study of weighted norm estimates. It is Lacey \cite{MR3625108} who firstly applied this technique to investigate probabilistic object, the discrete-time martingale transform.  Tanaka and Terasawa \cite{MR3004953} introduced the notion of
principal sets for martingales and applied it to provide a Sawyer type characterization  for a generalized Doob maximal
operator. The first two authors \cite{MR4244905,MR4125846} found
the conditional sparsity of principal sets and enhanced the quantitative weighted estimate in \cite{MR3004953}.
Based on this property, the authors \cite{ChenQuanJiaoWu} developed the theory of matrix-weighted martingales without recourse to the reverse H\"{o}lder inequality of matrix weights or any regularity conditions on the underlying filtration.  In the present paper, we further advance the conditional sparsity of principal sets and the technique of sparse domination to derive quantitative (or sharp) weighted estimates for martingale square functions, marking a substantial improvement on previous results in the literature.

As mentioned above, Hyt\"{o}nen, Petermichl and Volberg \cite{MR3936542} and Isralowitz \cite{MR4159390} only proved  Theorem \ref{main2} for martingales adapted to the dyadic filtration on $\mathbb R^n$, and it is nontrivial to generalize their results to even the case of dyadic lattice in $\mathbb R^n$. The extension to martingales associated with a general filtration on an abstract probability space  requires completely new ideas. In particular, a new construction of a dominating sparse operator,  tailoring to general martingales, was needed. Moreover,  the argument in \cite{MR3936542,MR4159390} relies on the reverse H\"{o}lder inequality for $A_\infty$ weights, which is unapplicable to our situation here.

The proof of Theorem \ref{main1} follows the same pattern with that of Theorem \ref{main2}. Due to this, we will just  outline the proof of Theorem \ref{main1} by pointing out the main difference.

The paper is organized as follows. The next section contains some basic properties of matrix $A_p$ weights and some fundamental results.
Section \ref{constr} is where we present the detailed construction of  principal sets and display some important  properties that will be used in the proof of our main results. In Section \ref{estimate-for-sparse}, we introduce the sparse operators $\mathcal T_{W,r}$ and  establish an $L_p$ estimate for $\mathcal T_{W,2}$.
The domination between $S_W$ and $\mathcal T_{W,2}$, and also the proof of Theorem \ref{main2} will be presented in Section \ref{Sec-Burk}. In the last section, we sketch the proof of Theorem \ref{main1}.

Throughout, we use $d \in \mathbb{N}$ to denote the dimension and the letter $c$ denotes a  constant, which (unless otherwise specified) probably depends on the dimension, but it may change from occurrence to occurrence. We write  $A \gtrsim_\delta B$ if there is a constant $c_\delta>0$ depending only on $\delta$  such that  $A \geq c_\delta B$ and we use the notation $A \approx_\delta B$ when $A \lesssim_\delta B$ and $A \gtrsim_\delta B$ hold simultaneously.

\medskip

\section{Preliminaries} \label{sec-matx-w}
In this section, we recall some basic properties of martingale matrix $A_p$ weights obtained in our precious work \cite{ChenQuanJiaoWu} and some facts on vector-valued martingales.

\subsection{Matrix $A_p$ weights}
Assume that $(\Omega, \mathcal{F},\mathbb P)$ is a probability space. For $1\leq d<\infty$ and $1\leq p<\infty$, the unweighted vector-valued Lebesgue space $L_p(\Omega; \mathbb C^d)$ consists of all measurable functions $f: \Omega \rightarrow \mathbb{C}^d$ such that
$$\|f\|_{L_p\left(\Omega;\mathbb C^d\right)}:=\left(\int_{\Omega}\|f(t)\|_{\mathbb{C}^d}^p d \mathbb P\right)^{1 / p}<\infty.$$
A matrix weight $W:\Omega\rightarrow \mathbb M_d(\mathbb C)$  is a $d\times d$  self-adjoint matrix-valued function on  $\Omega$ (with integrable entries) such that $W(x)$ is positive definite for almost all $x\in \Omega$. Given a matrix weight $W$ and $1\leq p<\infty$, we define the corresponding weighted Lebesgue space $L_p(W)$ to be the space of all measurable functions $f:\Omega\rightarrow\mathbb C^d$ such that the norm
$$\|f\|_{L_p(W)}:=\Big(\int_\Omega \|W^{\frac1p}(t)f(t)\|_{\mathbb C^d}^p d \mathbb P \Big)^{1/p}<\infty.$$	
Obviously, if $W\equiv\mbox{Id}_{\mathbb M_d}$, then $L_p(W)$ goes back to $L_p(\Omega; \mathbb C^d)$; and if $d=1$, then $L_p(W)$ reduces to the usual weighted Lebesgue spaces $L_p^w(\Omega)$.

Let $(\mathcal F_n)_{n\geq 1}$ be an increasing sequence of $\sigma$-subalgebras of $\mathcal F$ such that $\mathcal F=\sigma(\cup_{n\geq 1}\mathcal F_n)$ and let $(\mathbb E_n)_{n\geq 1}$ be the associated conditional expectations defined for scalar-valued functions. Given $n\geq 1$, we may apply the $n$-th conditional expectation $\mathbb E_n$ to $d$-dimensional vector-valued functions by allowing it to act separately on each coordinate function. More precisely, let $f=(f^1,f^2,\ldots,f^d)$ be a  vector-valued function defined on $\Omega$. The $n$-th conditional expectation of $f$ is given by
$$\mathbb E_nf=(\mathbb E_n (f^{1}),\mathbb E_n(f^2)\ldots,\mathbb E_n(f^d)).$$
For any vector-valued function $h\in L_1(\Omega;\mathbb C^d)$, it is clear that
$$\|\mathbb E_nh\|_{\mathbb C^d}\leq \mathbb E_n\|h\|_{\mathbb C^d},\quad n\geq 1.$$
This elementary fact of conditional expectations will be frequently used  in the sequel. Similarly, for any  matrix weight $W=(w^{i,j})_{1\leq i,j\leq d}$, the $n$-th conditional expectation of $W$ is given by
$$\mathbb E_n W=(\mathbb E_n(w^{i,j}))_{1\leq i,j\leq d}.$$

As we know, for scalar weights $w$, two objects are essential in defining $A_p$ conditions which are $[\mathbb E_n w]^{1/p}$ and  $[\mathbb E_n (w^{-p'/p})]^{1/{p'}}$. Here and  what follows, $p'$ denotes the conjugate number of $p$. Given a matrix weight $W$, we recall below the reducing matrix-valued functions $\widetilde W_n$, $\widehat W_n$ and some relevant properties proved in \cite[Lemma 2.4, Lemma 2.5]{ChenQuanJiaoWu}. Notably, these functions play the roles of $[\mathbb E_n w]^{1/p}$ and  $[\mathbb E_n (w^{-p'/p})]^{1/{p'}}$, respectively.

\begin{lemma}\label{construction w} Let $1\leq p<\infty$ and let $W$ be a matrix weight. There exists a sequence of matrix-valued functions $(\widetilde W_n)_{n\geq 1}$ such that for every $n$,  the entries of $\widetilde W_n$ are adapted to $ \mathcal F_n$ and
\begin{equation}\label{def-tilde-w}
\|\widetilde W_ne\|^p_{\mathbb C^d}\thickapprox_d \mathbb E_n\big(\|W^{\frac1p}e\|_{\mathbb C^d}^{p}\big),\quad \forall e\in \mathbb C^d.
\end{equation}
Moreover,  for all $n\in\mathbb N$ we have
	\begin{equation}\label{w-1-1}
 \mathbb E_n\big(\|W^{
				\frac1p}\widetilde W_n^{-1}\|^{p}\big)=\mathbb E_n\big(\|\widetilde W_n^{-1}W^{
				\frac1p}\|^{p}\big) \lesssim_{p,d} 1.
		\end{equation}
	\end{lemma}

\begin{lemma}\label{construction w2} Let $1 < p <\infty$ and let $W$ be a matrix weight. There exist a sequence of matrix-valued functions $\widehat W_n$ such that for every $n$, the  entries of $\widehat W_n$ are adapted to $ \mathcal F_n$ and
		\begin{equation}\label{w'-1-0}
			\|\widehat W_ne\|^{p'}_{\mathbb C^d}\thickapprox_d \mathbb E_n\big(\|W^{-\frac1p}e\|_{\mathbb C^d}^{p'}),\quad \forall e\in\mathbb C^d.
		\end{equation}
Moreover, for all $n\in\mathbb N$ we have
\begin{equation}\label{w'-1-2}
\mathbb E_n\big(\|W^{-\frac1p}\widehat W_n^{-1}\|^{p'}\big)=\mathbb E_n\big(\|\widehat W_n^{-1}W^{-\frac1p}\|^{p'})\lesssim_{p,d} 1.
\end{equation}
\end{lemma}

We now recall the definition of matrix $A_p$ condition introduced in \cite[Definition 2.6]{ChenQuanJiaoWu}.

\begin{definition}\label{Type-2}
Given $1<p<\infty$ and a matrix weight $W$, we say that $W$ satisfies the $A_p$ condition (denoted by $W\in A_p$) if
		\begin{equation}\label{equ-defi-2}
			[W]_{A_p}:=\sup_{n\geq 1}\|\widetilde W_n\widehat W_n\|^p<\infty.
		\end{equation}
Similarly, we say that $W$ satisfies the $A_1$ condition (denoted by $W\in A_1$) if
\begin{equation}\label{equ-defi-33}
[W]_{A_1}:=\sup_{n\geq 1}\|\widetilde{W}_nW^{-1}\|<\infty.
\end{equation}
The number $[W]_{A_p}$ is called the $A_p$ characteristic of $W$.
\end{definition}

The following result is taken from \cite[Corollary 2.11]{ChenQuanJiaoWu}, which provides equivalent descriptions of the characteristic of the matrix $A_p$ weight.

\begin{lemma} \label{equ-defi-ap}
Let $1<p<\infty$. Given  a matrix weight $W\in A_p$, we have
\begin{equation}\label{equi-chara}
[W]_{A_p}\approx_{p,d} \sup_{n\geq 1}\mathbb E_n\big(\|\widehat W_nW^{\frac{1}{p}}\|^{p}\big)\approx_{p,d}  \sup_{n\geq 1}\Big(\mathbb E_n\big(\|\widetilde W_nW^{-\frac{1}{p}}\|^{p'})\Big)^{\frac{p}{p'}}.
\end{equation}
\end{lemma}

Given a matrix weight $W\in A_p$,  we call $V=W^{-\frac{p'}{p}}$ the dual weight of $W$. It is easy to see that the  reducing matrix operators associated with $V$ can be taken as $\widehat V_n=\widetilde W_n$, $\widetilde V_n=\widehat W_n$ for all $n\in\mathbb N$. Therefore, we have the following basic property of dual weights.

\begin{lemma}\label{prop_AP} Let $1<p<\infty$. Given a matrix weight $W\in A_p$, the dual weight $V=W^{-\frac{p'}{p}}\in A_{p'}$ and
$[V]_{A_{p'}}\approx_{p,d} [W]_{A_{p}}^{p'-1}.$
\end{lemma}

\subsection{Vector-valued martingales}
A sequence $f=(f_n)_{n\geq 1}$  in $L_1(\Omega;\mathbb C^d)$ is called a $d$-dimensional $\mathbb C^d$-valued martingale with respect to  $(\mathcal F_n)_{n\geq 1}$ if $f_n$ is $\mathcal F_n$-measurable and
$$\mathbb E_{n}(f_{n+1})=f_{n},\quad n\geq 1.$$
For a $\mathbb C^d$-valued martingale $f = ( f_n)_{n\geq1}$, we define its martingale difference by setting $d_1f=f_1$ and
 $$d_nf =f_n-f_{n-1},\quad n\geq 2.$$
  For $1\leq p< \infty$, if $f_n\in L_p(\Omega;\mathbb C^d)$ for all $n \geq 1$, then $f=(f_n)_{n\geq 1}$ is called a $\mathbb C^d$-valued $L_p$ martingale. In this case, if we further have
$$\|f\|_{L_p(\Omega;\mathbb C^d)}:=\sup_{n\geq 1}\|f_n\|_{L_p(\Omega;\mathbb C^d)}<\infty$$
then $f$ is called a $\mathbb C^d$-valued $L_p$-bounded martingale. As is well known, for $1<p<\infty$, the martingale $f=(f_n)_{n\geq1}$ is $L_p$-bounded if and only if there exists a function $f_{\infty}\in L_p(\Omega;\mathbb C^d)$ such that $f_n=\mathbb E_n(f_\infty)$. Hence, in this case, we may identify a martingale $f$ with its final value $f_\infty$. We refer to \cite{Liupeide,MR394135} for more information on vector-valued martingales.

Similarly, for $1<p<\infty$ and a scalar weight $w$, one could define $L_p^w$-bounded martingales. The above  identification remain valid for $L_p^w$-bounded martingales; namely, for $1<p<\infty$, the martingale $f=(f_n)_{n\geq1}$ is $L_p^w$-bounded if and only if there exists a function $f_{\infty}\in L_p^w$ such that $f_n=\mathbb E_n(f_\infty)$.

Given  $1<p<\infty$ and a matrix weight $W$, we now recall a vector-valued maximal operator introduced in \cite[Page 22]{ChenQuanJiaoWu}:
$${M'_Wf=\sup_{n\geq 1}\mathbb E_n\|\widehat W_n^{-1}W^{-\frac1p}f\|_{\mathbb C^d},}$$
where $f=(f_n)_{n\geq1}$ is a $\mathbb C^d$-valued $L_p$-bounded martingale. The following  estimate for $M'_W$ will play a crucial role in the proof of our main result; see
\cite[Proposition 4.8]{ChenQuanJiaoWu} for the proof.

\begin{proposition}\label{RMprop}
Let $1<p<\infty$. Given  a matrix weight $W\in A_p$, we have $$\|M'_Wf\|_{L_p}\lesssim_{p,d}[W]_{A_p}^{\frac1{p(p-1)}} \|f\|_{L_p(\Omega;\mathbb C^d)},\quad \forall \, f\in L_p(\Omega;\mathbb C^d).$$
\end{proposition}
	
\medskip

\section{Principal sets and related properties}\label{constr}
In this section, we provide the detailed construction of principal sets that will be used in the proof of our main result. We start our consideration with a special family of functions. Let $1<p<\infty$ and let $W$ be a matrix weight. Given a $\mathbb C^d$-valued function $f$,  for $m>n$, we define
\begin{equation*}\label{function-related-to-square}
\gamma_1(n,m):=\frac{\left(\sum_{ i=n+1}^{  m}\|\widehat W_{n}^{-1}d_i( W^{-1/p}f)\|^2_{\mathbb C^d}\right)^{1/2}}{\mathbb E_{n}\|\widehat W_n^{-1}W^{-1/p}f\|_{\mathbb C^d}}\chi_{\{\mathbb E_{n}\|\widehat W_n^{-1}W^{-1/p}f\|_{\mathbb C^d}>0\}},
\end{equation*}

\begin{equation*}\label{function-related-to-maximal2}
\gamma_2(n,m):=\frac{\|\widehat W_n ^{-1}\mathbb E_m(W^{-1/p}f)\|_{\mathbb C^d}}{\mathbb E_{n}\|\widehat W_n^{-1} W^{-1/p}f\|_{\mathbb C^d}}\chi_{\{\mathbb E_{n}\|\widehat W_n^{-1} W^{-1/p}f\|_{\mathbb C^d}>0\}};
\end{equation*}and for $m\leq n$, we put $\gamma_i(n,m)=0$ for $i=1,$ $2$. Now set
\begin{equation*}\label{function-related-to-maximal3}
\gamma(n,m):=\max\left\{\gamma_1(n,m),\, \gamma_2(n,m)\right\}.
\end{equation*}
We collect below two useful properties of $\{\gamma(n,m)\}_{n,m\in \mathbb N}$.
\begin{lemma}\label{lem-two-cond}
The functions $\{\gamma(n,m)\}_{n,m\in \mathbb N}$
satisfy the following conditions:
\begin{enumerate}[\rm (i)]
\item For every $n,m\in \mathbb N$,  $\gamma(n,m)$ is adapted to $\mathcal F_{n\vee m}$.

\item For every $n\in\mathbb N$, there exists a constant $C_\gamma$ such that for every   $A\in\mathcal F_n$,
		$$\mathbb P\Big(A\cap\{\sup_{m>n}\gamma(n,m)>C_\gamma\}\Big)\leq\mathbb P\Big(A\cap\{\sup_{m>n}\gamma(n,m)\leq C_\gamma\}\Big). $$
\end{enumerate}
\end{lemma}

\begin{proof}
Indeed, (i) can be easily verified. Hence, we only check (ii) here. Fix $n\in\mathbb N$. For any $A\in\mathcal F_n$ and for arbitrary $\delta>0$, we have
\begin{align*}
&\mathbb P\Big(A\cap\{\sup_{m>n}\gamma(n,m)>\delta\}\Big)\\
&\quad=\mathbb P\Big(A\cap\big(\{\sup_{m>n}\gamma_1(n,m)>\delta\}\cup\{\sup_{m>n}\gamma_2(n,m)>\delta\}\big)\Big)\\
&\quad\leq \mathbb P\Big(A\cap\{\sup_{m>n}\gamma_1(n,m)>\delta\}\Big)+\mathbb P\Big(A\cap\{\sup_{m>n}\gamma_2(n,m)>\delta\}\Big)
\end{align*}
Let us first deal with the term involving $\gamma_1$. Set
$$\begin{cases}d_mF=0, &\mbox{if}\,\,m\leq n,\\
		d_mF=\frac{\widehat W^{-1}_{n}d_m(W^{-1/p}f)\chi_A}{\mathbb{E}_{n}\|\widehat{W}_{n}^{-1} W^{-1/p}f\|_{\mathbb C^d}}\chi_{\{\mathbb E_{n}\|\widehat W_n^{-1}W^{-1/p}f\|_{\mathbb C^d}>0\}},& \mbox{if}\,\, m>n.\end{cases}$$
Then  we have
\begin{align*}
\mathbb P\Big(A\cap\{\sup_{m>n}\gamma_1(n,m)>\delta\}\Big)&=\mathbb P\big(\{S(F)>\delta\}\big),
\end{align*}
where $S(F)=(\sum_{m}\|d_m F\|_{\mathbb C^d}^2)^{1/2}$. Applying the vector-valued analogue of weak-type estimate of square functions (cf. \cite{MR2853676}), we deduce that
$$
\begin{aligned}
\mathbb P\Big(A\cap\{\sup_{m>n}\gamma_1(n,m)>\delta\}\Big)&\leq \frac{\sqrt{e}}{\delta} \int_\Omega\|F\|_{\mathbb{C}^{d}} d \mathbb{P} \\
& \leq \frac{\sqrt{e}}{\delta} \int_{A} \frac{\|\widehat{W}_{n}^{-1} (W^{-1/p}f-\mathbb E_n(W^{-1/p}f))\|_{\mathbb{C}^{d}}}{\mathbb{E}_{n}\|\widehat{W}_{n}^{-1}W^{-1/p} f\|_{\mathbb{C}^{d}}} d \mathbb{P}\\
& \lesssim \frac{\sqrt{e}}{\delta} \int_{A} \frac{\|\widehat{W}_{n}^{-1} W^{-1/p} f\|_{\mathbb{C}^{d}}}{\mathbb{E}_{n}\|\widehat{W}_{n}^{-1} W^{-1/p} f\|_{\mathbb{C}^{d}}} d \mathbb{P}=\frac{\sqrt{e}}{\delta} \mathbb{P}(A).
\end{aligned}
$$
Hence, we may choose $C_{\gamma,1}$  such that
$$\mathbb P\Big(A\cap\{\sup_{m>n}\gamma_1(n,m)>C_{\gamma,1}\}\Big)\leq \frac14\mathbb P(A).$$

We turn to  deal with $\gamma_2$. Set
$$\tau=\inf\{m>n:\gamma_2(n,m)>\delta\}.$$
It follows from the Chebyshev  inequality that
\begin{equation*}\label{function-doob-first}
\begin{split}
&\mathbb P\Big(A\cap\{\sup_{m>n}\gamma_2(n,m)>\delta\}\Big)\\ &\leq \sum_{m> n}\int_{A\cap \{\tau=m\}}\frac{\gamma_2(n,m)}{\delta}d\mathbb P\\
		&=\sum_{m> n}\int_{A\cap \{\tau=m\}}\frac{\|\widehat  W^{-1}_n \mathbb E_m(W^{-1/p}f)\|_{\mathbb C^d}\chi_{\{\mathbb E_{n}\|\widehat  W^{-1}_nW^{-1/p}f\|_{\mathbb C^d}>0\}}}{{(\mathbb E_{n}\|\widehat  W^{-1}_nW^{-1/p}f\|_{\mathbb C^d}})\delta}d\mathbb P\\
		&\leq\sum_{m> n}\int_{A\cap \{\tau=m\}}\frac{\mathbb E_m(\|\widehat  W^{-1}_n W^{-1/p}f\|_{\mathbb C^d})\chi_{\{\mathbb E_{n}\|\widehat  W^{-1}_nW^{-1/p}f\|_{\mathbb C^d}>0\}}}{{(\mathbb E_{n}\|\widehat  W^{-1}_nW^{-1/p}f\|_{\mathbb C^d}})
		\delta}d\mathbb P  \\
	&=\int_{A}\frac{(\mathbb E_n\|\widehat  W^{-1}_n W^{-1/p}f\|_{\mathbb C^d})\chi_{\{\mathbb E_{n}\|\widehat  W^{-1}_nW^{-1/p}f\|_{\mathbb C^d}>0\}}}{{(\mathbb E_{n}\|\widehat  W^{-1}_nW^{-1/p}f\|_{\mathbb C^d}})
		\delta}d\mathbb P\leq \delta^{-1} \mathbb P(A).
	\end{split}
\end{equation*}
Thus, we are able to choose $C_{\gamma,2}$  depending only on $d$, $p$ such that \begin{align*}
	\mathbb P\Big(A\cap\{\sup_{m>n}\gamma_2(n,m)>C_{\gamma,2}\}\Big)&\leq \frac1{4}\mathbb P(A).
\end{align*}

Setting $C_\gamma=\max\{C_{\gamma,1},C_{\gamma,2}\}$ and putting all the above together, we get
\begin{align*}
&\mathbb P\Big(A\cap\{\sup_{m>n}\gamma(n,m)>C_\gamma\}\Big)\\
&\leq \mathbb P\Big(A\cap\{\sup_{m>n}\gamma_1(n,m)>C_{\gamma,1}\}\Big)+\mathbb P\Big(A\cap\{\sup_{m>n}\gamma_2(n,m)>C_{\gamma,2}\}\Big)\\
	&\leq\frac1{2}\mathbb P(A),
\end{align*}
which implies that
$$\mathbb P\Big(A\cap\{\sup_{m>n}\gamma(n,m)>C_\gamma\}\Big)\leq\mathbb P\Big(A\cap\{\sup_{m>n}\gamma(n,m)\leq C_\gamma\}\Big).$$
This verifies that $\{\gamma(n,m)\}_{n,m\in \mathbb N}$ satisfies  the condition  (ii).
\end{proof}

We now introduce the principal sets associated with $\{\gamma(n,m)\}_{n,m\in \mathbb N}$. For any fixed $n,$ define
	$$
	\tau_{1}:=\inf \left\{\ell >n: \gamma(n,\ell)>0\right\}.
	$$
	For every $j>n$, denote by
	$$P_{1, j}:=\{\tau_1=j\}$$
	and
	$$\kappa_{1}\left(P_{1, j}\right):=n,\quad\kappa_{2}\left(P_{1, j}\right):=j.$$
We call $\mathcal{P}_{1}=\left\{P_{1, j}:~j>n\right\}$ the first generation of principal sets. Without risk of confusion, we will write $P_1$ in place of $P_{1, j}$ in the sequel.
	
Let us get the second generation. For any $P_1\in\mathcal P_1$, define
	$$
	\tau_{P_{1}}:=\inf \left\{{\ell>\kappa_{2}\left(P_{1}\right)}: \gamma(\kappa_2(P_1),\ell) \chi_{P_{1}}>C_\gamma\right \}.
	$$
For every $j>\kappa_2(P_{1})$, denote by
	$$P_{2,j}:=\{\tau_{P_{1}}=j\}$$
	and $$\kappa_{1}\left(P_{2,j}\right):=\kappa_{2}\left(P_{1}\right),\quad\kappa_{2}\left(P_{2,j}\right):=j.$$
We call $\mathcal{P}(P_1):=\left\{P_{2, j}:~j> \kappa_2(P_{1})\right\}$ the collection of the principal sets generated by $P_1.$ Moreover, letting $\mathcal{P}_{2}=\bigcup_{P_1 \in \mathcal{P}_{1}} \mathcal{P}(P_1),$ we call $\mathcal{P}_{2}$ the second generation of principal sets.
	
The next generations can be defined by induction. For any fixed $P_{m}\in\mathcal P_{m},$ define
	$$
	\begin{aligned}
		\tau_{P_{m}}
		:=\inf \{&l>\kappa_{2}\left(P_{m}\right): \gamma(\kappa_{2}\left(P_{m}\right),\ell) \chi_{P_{m}}
		>C_\gamma\}.
	\end{aligned}
	$$
	For every $j>\kappa_2(P_{m})$, we denote
	$$P_{m+1,j}:=\{\tau_{P_{m}}=j\}$$
	and	$$\kappa_{1}\left(P_{m+1,j}\right):=\kappa_{2}\left(P_{m}\right),\quad \kappa_{2}\left(P_{m+1,j}\right):=j.$$
Let $\mathcal{P}(P_m)=\left\{P_{m+1, j}:~j> \kappa_2(P_{m})\right\}$ and we call it the collection of the principal sets generated by $P_m.$ Moreover, letting $\mathcal{P}_{m+1}=\bigcup_{P_m \in \mathcal{P}_{m}} \mathcal{P}(P_m),$
we call $\mathcal{P}_{m+1}$ the $(m+1)\hbox{-th}$ generation of principal sets.
	
Define the collection of principal sets $\mathcal{P}$ by
	$$
	\mathcal{P}:=\bigcup_{m=1}^{\infty} \mathcal{P}_{m}.
	$$
Then, by \cite[Proposition 4.1]{ChenQuanJiaoWu}, we find the following properties of principal sets.

\begin{lemma}\label{prop-principal-set} For the principal sets introduced above, we have
\begin{enumerate}[\rm (a)]
\item \label{Pro-1} The sets $E(P):=P\cap\{\tau_{P}=\infty\}$, where $P\in \mathcal P$, are disjoint.
\smallskip

\item  \label{Pro-2} For every $P\in\mathcal P$, $P\in \mathcal F_{\kappa_2(P)}$.
\smallskip

\item  \label{Pro-3} For  every $P\in \bigcup_{m=2}^{\infty} \mathcal{P}_{m}$,
$$\Big(\sum_{\kappa_1(P)<i\leq \kappa_2(P)-1}\|\widehat W_{\kappa_1(P)}^{-1}d_i(W^{-\frac1p}f)\|_{\mathbb C^d}^2\Big)^{\frac12}\leq C_\gamma\mathbb E_{\kappa_1(P)}\|\widehat W_{\kappa_1(P)}^{-1}W^{-\frac1p}f\|_{\mathbb C^d}\,\,\, \mbox{on}\,\,\,P$$
and
$$\sup_{\kappa_1(P)<i\leq \kappa_2(P)-1}\|\widehat W_{\kappa_1(P)}^{-1} \mathbb E_{i}(W^{-\frac{1}{p}}f)\|_{\mathbb C^d}\leq C_\gamma\mathbb E_{\kappa_1(P)}\|\widehat  W_{\kappa_1(P)}^{-1}W^{-\frac{1}{p}}f\|_{\mathbb C^d}\,\,\, \mbox{on}\,\,\, P.$$

\item  \label{Pro-4} For every $P\in \bigcup_{m=2}^{\infty} \mathcal{P}_{m}$, $$\Big(\sum_{i>\kappa_2(P)}\|\widehat W_{\kappa_2(P)}^{-1}d_i(W^{-\frac1p}f)\|_{\mathbb C^d}^2\Big)^{\frac12}\leq C_\gamma\mathbb E_{\kappa_2(P)}\|\widehat W_{\kappa_2(P)}^{-1}W^{-\frac1p}f\|_{\mathbb C^d}\quad \mbox{on}\quad E(P)$$
and
$$\sup_{i> \kappa_2(P)}\|\widehat  W_{\kappa_2(P)}^{-1} \mathbb E_{i}(W^{-1/p}f)\|_{\mathbb C^d}\leq C_\gamma\mathbb E_{\kappa_2(P)}\|\widehat  W_{\kappa_2(P)}^{-1}W^{-\frac{1}{p}}f\|_{\mathbb C^d}\quad \mbox{on}\quad E(P).$$

\smallskip

\item  \label{Pro-5} For  every $P\in\bigcup_{m=2}^\infty\mathcal P_m$, we have
    $$\mathbb P(P)\leq 2\mathbb P(E(P))$$ and $$\chi_P\leq 2\mathbb E_{\kappa_2(P)}(\chi_{E(P)}).$$
\smallskip

\item \label{Pro-6} $\lim_{N\to \infty} \mathbb P(\bigcup_{P\in{\mathcal P}_N} P)=0.$
	\end{enumerate}
\end{lemma}

\medskip

\section{An $L_p$ estimate for the sparse operator $\mathcal T_{W,2}$}\label{estimate-for-sparse}
The purpose of this section is to introduce the sparse operators $\mathcal T_{W,r}$ and  establish an $L_p$ estimate for $\mathcal T_{W,2}$. Let $W$ be a matrix weight. Fix $n\in\mathbb N$. Consider the principal sets constructed in the preceding section. For $r\geq 1$, we introduce the sparse operators $\mathcal T_{W,r}$ as follows:
\begin{align*}
	&\mathcal T_{W,r}f:=\Big(\sum_{P \in \mathcal{P}}\|W^{\frac{1}{p}} \widehat{W}_{\kappa_{2}(P)}\|^r \Big(\mathbb{E}_{\kappa_{2}(P)}\|\widehat{W}_{\kappa_{2}(P)}^{-1} W^{-\frac1p}f\|_{\mathbb{C}^{d}}\Big)^r\chi_{P}\Big)^{\frac1r}.
\end{align*}
The main result of this section reads as follows.

\begin{proposition}\label{5.4}
Let  $1<p<\infty$.  Given a matrix weight $W\in A_p$, we have
	\begin{equation*}\label{Lem-main}
		\left\|\mathcal T_{W,2}f\right\|_{L_{p}} \lesssim_{p,d}[W]_{A_p}^{\max\{\frac12+\frac1{p(p-1)},\frac1{p-1}\}} \|f\|_{L_p(\Omega;\mathbb C^d)},\quad\forall f\in L_p(\Omega;\mathbb C^d).
	\end{equation*}
\end{proposition}

In preparation for Proposition \ref{5.4}, we need several lemmas. The first lemma shows that the operator $\mathcal T_{W,p}$ is bounded on $L_p$ for $1<p<\infty$.

\begin{lemma}\label{5.4'}
Let  $1<p<\infty$.  If $W\in A_p$, then
\begin{equation*}\label{Lem-main}
		\left\|\mathcal T_{W,p}f\right\|_{L_p} \lesssim_{p,d}[W]_{A_p}^{\frac 1{p-1}} \|f\|_{L_p(\Omega;\mathbb C^d)},\quad\forall f\in L_p(\Omega;\mathbb C^d).
	\end{equation*}
\end{lemma}

\begin{proof}
By Property \eqref{Pro-2} in Lemma \ref{prop-principal-set}, we have $P\in\mathcal F_{\kappa_2(P)}$ for every $P\in\mathcal P$. Therefore,
	\begin{align*}\left\| \mathcal T_{W,p}f\right\|_{L_p}^p &=\sum_{P \in \mathcal{P}} \int_{\Omega}\mathbb E_{\kappa_2(P)}\big(\big\|W^{\frac{1}{p}} \widehat{W}_{\kappa_{2}\left(P\right)}\big\|^p\big) \Big(\mathbb{E}_{\kappa_{2}\left(P\right)}\big\|\widehat{W}_{\kappa_{2}\left(P\right)}^{-1} W^{-\frac1p}f\big\|_{\mathbb{C}^{d}}\Big)^p\chi_{P}d\mathbb P.
\end{align*}
Since $W\in A_p$, it follows from \eqref{equi-chara} that
\begin{align*}
\left\| \mathcal T_{W,p}f\right\|_{L_p}^p &\lesssim_{p,d}[W]_{A_p}\sum_{P \in \mathcal{P}} \int_{\Omega}\Big(\mathbb{E}_{\kappa_{2}\left(P\right)}\big\|\widehat{W}_{\kappa_{2}\left(P\right)}^{-1} W^{-\frac1p}f\big\|_{\mathbb{C}^{d}}\Big)^p \chi_{P}d\mathbb P\\
		&=[W]_{A_p}\sum_{i=2}^\infty\sum_{P \in \mathcal{P}_i} \int_{\Omega}\Big(\mathbb{E}_{\kappa_{2}\left(P\right)}\big\|\widehat{W}_{\kappa_{2}\left(P\right)}^{-1} W^{-\frac1p}f\big\|_{\mathbb{C}^{d}}\Big)^p \chi_{P}d\mathbb P\\&\quad+[W]_{A_p}\sum_{P \in \mathcal{P}_1} \int_{\Omega}\Big(\mathbb{E}_{\kappa_{2}\left(P\right)}\big\|\widehat{W}_{\kappa_{2}\left(P\right)}^{-1} W^{-\frac1p}f\big\|_{\mathbb{C}^{d}}\Big)^p \chi_{P}d\mathbb P=: \mbox{I}+\mbox{II}.
\end{align*}
Note that principal sets from the same generation are disjoint. Thus,
\begin{align*}
\mbox{II}\leq [W]_{A_p} \|M_{W}'f\|_{L_p}^p.
\end{align*}
One the other hand, by Property \eqref{Pro-5}  in Lemma \ref{prop-principal-set}, we get
\begin{align*}
\mbox{I} &\lesssim [W]_{A_p}\sum_{i=2}^\infty\sum_{P \in \mathcal{P}_i} \int_{\Omega} \Big(\mathbb{E}_{\kappa_{2}\left(P\right)}\big\|\widehat{W}_{\kappa_{2}\left(P\right)}^{-1} W^{-\frac1p}f\big\|_{\mathbb{C}^{d}}\Big)^p \mathbb E_{\kappa_{2}(P)}(\chi_{E(P)})d\mathbb P\\
&=[W]_{A_p}\sum_{i=2}^\infty\sum_{P \in \mathcal{P}_i} \int_{\Omega} \Big(\mathbb{E}_{\kappa_{2}\left(P\right)}\big\|\widehat{W}_{\kappa_{2}\left(P\right)}^{-1} W^{-\frac1p}f\big\|_{\mathbb{C}^{d}}\Big)^p \chi_{E(P)} d\mathbb P.
\end{align*}
Hence, it follows from  Property \eqref{Pro-1} in Lemma \ref{prop-principal-set} that
\begin{align*}
\mbox{I} &\lesssim [W]_{A_p} \|M_{W}'f\|_{L_p}^p
\end{align*}
Putting the above estimates together and then applying Proposition	\ref{RMprop}, we conclude that
\begin{align*}
\left\| \mathcal T_{W,p}f\right\|_{L_p}^p &\lesssim_{p,d}[W]_{A_p}\|M_{W}'f\|_{L_p}^p \leq [W]_{A_p}^{\frac{p}{p-1}} \|f\|_{L_p(\Omega;\mathbb C^d)}^p,
\end{align*}
which is the desired result.
\end{proof}

The lemma below could be viewed as one type of weighted maximal inequality.

\begin{lemma}\label{another-maximal}
Let $1<p<\infty$. Given a matrix weight $W\in A_p$, we have
\begin{equation}\label{something-remain-2}
		\int_{\Omega}\sup_n \left[ \mathbb E_{n}\left(\big\|W^{\frac{1}{p}} \widehat{W}_{n}\big\|g\right)\right]^{p'}d\mathbb P\lesssim_{p,d}[W]_{A_p}^{p'}\|g\|_{L_{p'}}^{p'},\quad \forall\, 0\leq g\in L_{p'}.
	\end{equation}
\end{lemma}

\begin{proof}
Let $(e_1,\cdots,e_d)$ be an orthonormal basis of $\mathbb C^d.$ Clearly, for  $0\leq g\in L_{p'}$, we have
	$$\big\|W^{\frac{1}{p}} \widehat{W}_{n}\big\|g=\big\| \widehat{W}_{n}W^{\frac{1}{p}}\big\|g  \leq \sum_{i=1}^d\big\| \widehat{W}_{n}W^{\frac{1}{p}}ge_i\big\|_{\mathbb C^d}.$$
	Hence, it follows from the $A_p$ condition that
	\begin{equation}\label{something-remain}
\begin{split}
&\int_{\Omega}\sup_n \left( \mathbb E_{n}\left(\big\|W^{\frac{1}{p}} \widehat{W}_{n}\big\|g\right)\right)^{p'}d\mathbb P\\&\lesssim_d \sum_{i=1}^d\int_{\Omega}\sup_n\left(\mathbb E_n\big\| \widehat{W}_{n}\widetilde{W}_{n}\widetilde{W}_{n}^{-1}W^{\frac{1}{p}}ge_i\big\|_{\mathbb C^d}\right)^{p'}d\mathbb P\\
	 &\leq[W]_{A_p}^{\frac1{p-1}}\sum_{i=1}^d\int_{\Omega}\sup_n\left(\mathbb E_n\big\| \widetilde{W}^{-1}_{n}W^{\frac{1}{p}}ge_i\big\|_{\mathbb C^d}\right)^{p'}d\mathbb P.
		\end{split}
	\end{equation}
Now, denote by $V$ the dual weight of $W$ (i.e., $V=W^{-\frac{p'}{p}}$). By Lemma \ref{prop_AP}, $V\in A_{p'}$ and $[V]_{A_{p'}}\approx_{d,p}[W]_{A_p}^{\frac{p'}{p}}$. Moreover, $$\widehat V_n=\widetilde W_n,\quad \quad \widetilde V_n=\widehat W_n. $$
		Therefore,  for each $i\in\{1,2,\ldots d\}$, we deduce that
		\begin{align*}\int_{\Omega}\sup_n\left(\mathbb E_n\big\| \widetilde{W}^{-1}_{n}W^{\frac{1}{p}}ge_i\big\|_{\mathbb C^d}\right)^{p'}d\mathbb P
			&= \int_{\Omega}\sup_n\left(\mathbb E_n\big\|\widehat{V}^{-1}_{n}V^{-\frac1{p'}}ge_i\big\|_{\mathbb C^d}\right)^{p'}d\mathbb P\\
			&=\|M'_{V}(ge_i)\|_{L_{p'}}^{p'}.\end{align*}
		Applying Proposition \ref{RMprop} (for $p'$, $V$
		and $ge_i$), we obtain
		\begin{align*}
			\int_{\Omega}\sup_n\left(\mathbb E_n\big\| \widetilde{W}^{-1}_{n}W^{\frac{1}{p}}ge_i\big\|_{\mathbb C^d}\right)^{p'}d\mathbb P
			&{\lesssim_{p,d}}[V]_{A_{p'}}^{\frac{1}{p' -1}}\|ge_i\|_{L_{p'}(\Omega;\mathbb C^d)}^{p'}=[W]_{A_p}\|g\|_{L_{p'}}^{p'}.
		\end{align*}
		Substituting this into \eqref{something-remain}, we obtain the desired estimate \eqref{something-remain-2}.
\end{proof}

By Lemma \ref{another-maximal} and duality, one can establish the following  $L_p$ estimate for $\mathcal T_{W,1}$.

\begin{lemma}\label{5.4''}
Let $1<p<\infty$. Given a matrix weight $W\in A_p$, we have
\begin{equation*}\label{Lem-main}
\left\|\mathcal T_{W,1}f\right\|_{L_{p}} \lesssim_{p,d}[W]_{A_p}^{1+\frac{1}{p(p-1)}} \|f\|_{L_p(\Omega;\mathbb C^d)},\quad\forall f\in L_p(\Omega;\mathbb C^d).
\end{equation*}
\end{lemma}
\begin{proof}
By duality and Property \eqref{Pro-2}, we have
	\begin{align*}
	&\left\| \mathcal T_{W,1}f\right\|_{L_{p}}\\
		&= \sup_{\substack{ g \footnotesize{\mbox{ is positive}}\\\|g\|_{L_{p'}}\leq 1}}\int_{\Omega}\sum_{P \in \mathcal{P}}\big\|W^{\frac{1}{p}} \widehat{W}_{\kappa_{2}\left(P\right)}\big\| \mathbb{E}_{\kappa_{2}\left(P\right)}\Big(\big\|\widehat{W}_{\kappa_{2}\left(P\right)}^{-1} W^{-\frac1p}f\big\|_{\mathbb{C}^{d}} \Big)\chi_{P} gd\mathbb P\\
		&=\sup_{\substack{ g \footnotesize{\mbox{ is positive}}\\\|g\|_{L_{p'}}\leq 1}}\sum_{P \in \mathcal{P}}\int_{\Omega}\mathbb E_{\kappa_2(P)}\Big(\big\|W^{\frac{1}{p}} \widehat{W}_{\kappa_{2}\left(P\right)}\big\|g\Big) \mathbb{E}_{\kappa_{2}\left(P\right)}\Big(\big\|\widehat{W}_{\kappa_{2}\left(P\right)}^{-1} W^{-\frac1p}f\big\|_{\mathbb{C}^{d}} \chi_{P} \Big)d\mathbb P\\
&\leq \sup_{\substack{g \footnotesize{\mbox{ is positive}}\\\|g\|_{L_{p'}}\leq 1}}\Big[\sum_{P\in\mathcal P_{1}}\int_{P} \mathbb{E}_{\kappa_{2}\left(P\right)}\Big(\big\|\widehat{W}_{\kappa_{2}\left(P\right)}^{-1}W^{-\frac1p} f\big\|_{\mathbb{C}^{d}}\Big) \mathbb E_{\kappa_2(P)}\left(\big\|W^{\frac{1}{p}} \widehat{W}_{\kappa_{2}\left(P\right)}\big\|g\right)d\mathbb P\\
		&\qquad +\sum_{P\in\bigcup_{i=2}^\infty\mathcal P_{i}}\int_{P} \mathbb{E}_{\kappa_{2}\left(P\right)}\Big(\big\|\widehat{W}_{\kappa_{2}\left(P\right)}^{-1}W^{-\frac1p} f\big\|_{\mathbb{C}^{d}}\Big) \mathbb E_{\kappa_2(P)}\left(\big\|W^{\frac{1}{p}} \widehat{W}_{\kappa_{2}\left(P\right)}\big\|g\right)d\mathbb P\Big].
	\end{align*}
According to Property \eqref{Pro-5}, $\chi_P\leq 2\mathbb E_{\kappa_2(P)}(\chi_{E(P)})$ for  every $P\in\bigcup_{m=2}^\infty\mathcal P_m$. Therefore,
\begin{align*}
&\left\| \mathcal T_{W,1}f\right\|_{L_{p}}\\
		&\lesssim \sup_{\substack{g \footnotesize{\mbox{ is positive}}\\\|g\|_{L_{p'}}\leq 1}}\Big[\sum_{P\in\mathcal P_{1}}\int_{P} \mathbb{E}_{\kappa_{2}\left(P\right)}\Big(\big\|\widehat{W}_{\kappa_{2}\left(P\right)}^{-1}W^{-\frac1p} f\big\|_{\mathbb{C}^{d}}\Big) \mathbb E_{\kappa_2(P)}\left(\big\|W^{\frac{1}{p}} \widehat{W}_{\kappa_{2}\left(P\right)}\big\|g\right)d\mathbb P\\
		&\qquad +\sum_{P\in\bigcup_{i=2}^\infty\mathcal P_{i}}\int_{E(P)} \mathbb{E}_{\kappa_{2}\left(P\right)}\Big(\big\|\widehat{W}_{\kappa_{2}\left(P\right)}^{-1}W^{-\frac1p} f\big\|_{\mathbb{C}^{d}}\Big) \mathbb E_{\kappa_2(P)}\left(\big\|W^{\frac{1}{p}} \widehat{W}_{\kappa_{2}\left(P\right)}\big\|g\right)d\mathbb P\Big].
	\end{align*}
In view of Property \eqref{Pro-1} and H\"older's inequality, we obtain
	\begin{align*}
		\left\| \mathcal T_{W,1}f\right\|_{L_{p}}
		&\lesssim  \sup_{\substack{g \footnotesize{\mbox{ is positive}}\\\|g\|_{L_{p'}}\leq 1}}\int_{\Omega}\sup_n \mathbb{E}_{n}\Big(\big\|\widehat{W}_n^{-1} W^{-\frac1p}f\big\|_{\mathbb{C}^{d}}\Big) \mathbb E_{n}\left(\big\|W^{\frac{1}{p}} \widehat{W}_{n}\big\|g\right)d\mathbb P
		\\
		&\leq \|M_W'f\|_{L_p}\sup_{\substack{ g \footnotesize{\mbox{ is positive}}\\\|g\|_{L_{p'}}\leq 1}}\left(\int_{\Omega}\sup_n \left[\mathbb E_{n}\left(\big\|W^{\frac{1}{p}} \widehat{W}_{n}\big\|g\right)\right]^{p'}d\mathbb P\right)^{\frac{1}{p'}},
	\end{align*}
Now, applying Proposition \ref{RMprop} and Lemma \ref{another-maximal}, we get
	\begin{equation*}
		\left\| \mathcal T_{W,1}f\right\|_{L_{p}}\lesssim_{p,d}[W]_{A_p}^{\frac1{p(p-1)}+1}\|f\|_{L_p(\Omega;\mathbb C^d)}.
	\end{equation*}
The assertion is verified.
\end{proof}

We are now ready to show Proposition \ref{5.4}.

\begin{proof}[Proof of Proposition \ref{5.4}] The proof shall be divided into two cases: $1<p\leq 2$ and $2<p<\infty$. The case $1<p\leq 2$ is rather easy. Indeed, since $\ell_p\subset \ell_2$, it follows that
\begin{align*}
\mathcal T_{W,2} f \leq  \mathcal T_{W,p}f.
	\end{align*}
Thus, by Lemma \ref{5.4'}, we get the desired result.
	
The case $2<p<\infty$ is more involved.  Set $\theta=\frac{p}{2p-2}$. Then $\frac12=(1-\theta)+\frac{\theta}{p}$. By H\"older's inequality, we obtain
	\begin{align*}
			\mathcal T_{W,2}f\leq (\mathcal T_{W,1}f)^{1-\theta}(\mathcal T_{W,p}f)^{\theta}. 		\end{align*}
Using H\"older's inequality again, we get
		\begin{align}\label{es-Tw2}
			\mathcal \|T_{W,2}f\|_{L_p}&\leq\|\mathcal T_{W,1}f\|_{L_p}^{1-\theta}\|\mathcal T_{W,p}f\|_{L_p}^{\theta}.
		\end{align}
Then, it follows from  Lemma \ref{5.4''} and  Lemma \ref{5.4'} that
\begin{align*}
			\|\mathcal T_{W,2}f\|_{L_p} &\lesssim_{d,p}\Big([W]_{A_p}^{1+\frac1{p(p-1)}}\|f\|_{L_p(\Omega;\mathbb C^d)}\Big)^{1-\theta}\Big([W]_{A_p}^{\frac 1{p-1}}\|f\|_{L_p(\Omega;\mathbb C^d)}\Big)^{\theta}\\
&=[W]_{A_p}^{\frac12+\frac1{p(p-1)}}\|f\|_{L_p(\Omega;\mathbb C^d)}.
\end{align*}
The proof is complete.
\end{proof}

\medskip

\section{Sparse domination of $S_W$ and Proof of Theorem \ref{main2}}\label{Sec-Burk}
The goal of this section is to establish the following sparse domination of $S_W$. This result, together with Proposition \ref{5.4}, immediately implies Theorem \ref{main2}.  From now on, we assume that $n=0$ and consider the associated principal sets and sparse operators (which will be still denoted by $\mathcal T_{W,r}$).

\begin{proposition}\label{Sdominition} Let $\mathcal T_{W,2}$ be defined as above. We have
$$S_W f \lesssim_{p,d} \mathcal T_{W,2} f,\quad\forall f\in L_1(\Omega;\mathbb C^d).$$
\end{proposition}

Before giving the proof, we state and verify several lemmas.
For $f\in L_1(\Omega;\mathbb C^d)$ and $m\in \mathbb N$, set
\begin{equation}\label{def-b}
b_m:=\Big(\sum_{P_{m} \in \mathcal{P}_{m}}  \sum_{k>\kappa_2(P_m)}\|W^{\frac1p}d_k(W^{-\frac1p}f)\|_{\mathbb C^d}^2\chi_{P_m}\Big)^{\frac12}.
\end{equation}
The lemma below presents an essential iterative formula for $(b_m)_{m\in\mathbb N}$, which constitutes a key ingredient of the proof of Proposition \ref{Sdominition}.

\begin{lemma}\label{reiteraton}
For each $m\in  \mathbb N$, we have	
\begin{equation}\label{claimsq}
\begin{split}
b_m^2&\leq b_{m+1}^2+2C_\gamma^2 \sum_{P_{m} \in \mathcal{P}_{m}}\|W^{\frac{1}{p}} \widehat{W}_{\kappa_{2}(P_{m})}\|^2 \Big(\mathbb{E}_{\kappa_{2}(P_{m})}\|\widehat{W}_{\kappa_{2}(P_{m})}^{-1} W^{-\frac1p}f\|_{\mathbb{C}^{d}}\Big)^2\chi_{P_m}\\
&\qquad + \sum_{P_{m+1} \in \mathcal{P}_{m+1}}\|W^{\frac{1}{p}} \widehat{W}_{\kappa_{1}(P_{m+1})}\|^2 \Big(\mathbb{E}_{\kappa_{1}(P_{m+1})}\|\widehat{W}_{\kappa_{1}(P_{m+1})}^{-1} W^{-\frac1p}f\|_{\mathbb{C}^{d}}\Big)^2\chi_{P_{m+1}}
\end{split}
\end{equation}
where $C_\gamma$ is the constant given in Lemma \ref{lem-two-cond}.
Furthermore, given an integer $N>1$, we have
\begin{equation}\label{after-iteration}
\begin{split}
b_1^2&\leq b_{N}^2+(2C_\gamma^2+1) \sum_{m=1}^{N}\sum_{P_{m} \in \mathcal{P}_{m}}\|W^{\frac{1}{p}} \widehat{W}_{\kappa_{2}(P_{m})}\|^2 \Big(\mathbb{E}_{\kappa_{2}(P_{m})}\|\widehat{W}_{\kappa_{2}(P_{m})}^{-1} W^{-\frac1p}f\|_{\mathbb{C}^{d}}\Big)^2\chi_{P_m}.
\end{split}
\end{equation}
\end{lemma}

\begin{proof}
Let us first check \eqref{claimsq}. For each $m\in  \mathbb N$, we write
	$$b_m^2=b_{m,1}^2+b_{m,2}^2$$
	with
$$b_{m,1}:=\Big(\sum_{P_m\in\mathcal P_m}\sum_{j>\kappa_2(P_m)}\|W^{\frac1p}d_j(W^{-\frac1p}f)\|^2_{\mathbb C^d}\chi_{E(P_m)}\Big)^{\frac12};$$
	$$b_{m,2}:=\Big(\sum_{P_m\in\mathcal P_m}\sum_{j>\kappa_2(P_m)}\|W^{\frac1p}d_j(W^{-\frac1p}f)\|^2_{\mathbb C^d}\chi_{\{\tau_{P_m}<\infty\}}\Big)^{\frac12}.$$
	For the term $b_{m,1}$, we deduce from Property \eqref{Pro-4} that
	\begin{align*}
		b_{m,1}^2&=\sum_{P_m\in\mathcal P_m}\sum_{j>\kappa_2(P_m)}\|W^{\frac1p}\widehat W_{\kappa_2(P_m)}\widehat W^{-1}_{\kappa_2(P_m)}d_j(W^{-\frac1p}f)\|^2_{\mathbb C^d}\chi_{E(P_m)}\\
		&\leq \sum_{P_m\in\mathcal P_m}\sum_{j>\kappa_2(P_m)}\|W^{\frac1p}\widehat W_{\kappa_2(P_m)}\|^2\|\widehat W^{-1}_{\kappa_2(P_m)}d_j(W^{-\frac1p}f)\|^2_{\mathbb C^d}\chi_{E(P_m)}\\
		&\leq C_\gamma^2\sum_{P_m\in\mathcal P_m}\|W^{\frac1p}\widehat W_{\kappa_2(P_m)}\|^2\Big(\mathbb E_{\kappa_2(P_m)}\|\widehat W_{\kappa_2(P_m)}^{-1}W^{-\frac1p}f\|_{\mathbb C^d}\Big)^2\chi_{E(P_m)}.
	\end{align*}
	Now we turn to estimate $b_{m,2}$. Note that for any fixed $P_m\in\mathcal P_{m}$,
	\begin{equation}\label{observation-a-2}\{\tau_{P_{m}}<\infty\}=\bigcup_{P\in \mathcal P(P_m)}P = \sum_{P\in \mathcal P(P_m)}P.\end{equation}
Combining this with the fact $\kappa_1(P_{m+1})=\kappa_2(P_m)$, we have
	\begin{equation}\label{bm2}
		\begin{split}
			b_{m,2}^2&=\sum_{P_m\in\mathcal P_m}\sum_{P_{m+1}\in\mathcal P_{m+1}(P_m)}\sum_{k>\kappa_2(P_m)}\|W^{\frac1p}d_k(W^{-\frac1p}f)\|_{\mathbb C^d}^2 \chi_{P_{m+1}}\\
			&=\sum_{P_{m+1}\in\mathcal P_{m+1}}\sum_{k>\kappa_1(P_{m+1})}\|W^{\frac1p}d_k(W^{-\frac1p}f)\|_{\mathbb C^d}^2 \chi_{P_{m+1}}\\
			&\leq \sum_{P_{m+1}\in\mathcal P_{m+1}} \sum_{\kappa_1(P_{m+1})<k\leq \kappa_2(P_{m+1})}\|W^{\frac1p}d_k(W^{-\frac1p}f)\|_{\mathbb C^d}^2\chi_{P_{m+1}}\\
			&\qquad+\sum_{P_{m+1}\in\mathcal P_{m+1}}\sum_{i>\kappa_2(P_{m+1})}\|W^{\frac1p}d_i(W^{-\frac1p}f)\|^2_{\mathbb C^d}\chi_{P_{m+1}}\\ 	&= \sum_{P_{m+1}\in\mathcal P_{m+1}} \sum_{\kappa_1(P_{m+1})<k\leq \kappa_2(P_{m+1})}\|W^{\frac1p}d_k(W^{-\frac1p}f)\|_{\mathbb C^d}^2\chi_{P_{m+1}}+b_{m+1}^2.
		\end{split}
	\end{equation}
On one hand, by virtue of Property \eqref{Pro-3}, $\kappa_1(P_{m+1})=\kappa_2(P_m)$ and \eqref{observation-a-2}, we get
	\begin{equation}\label{estimate-b2-1}
	\begin{split}
		&\sum_{P_{m+1}\in\mathcal P_{m+1}} \sum_{\kappa_1(P_{m+1})<k< \kappa_2(P_{m+1})}\|W^{\frac1p}d_k(W^{-\frac1p}f)\|_{\mathbb C^d}^2\chi_{P_{m+1}}\\
		&\leq \sum_{P_{m+1}\in\mathcal P_{m+1}} \sum_{\kappa_1(P_{m+1})<k< \kappa_2(P_{m+1})}\|W^{\frac{1}{p}} \widehat{W}_{\kappa_1(P_{m+1})}\|^2 \|\widehat{W}_{\kappa_1(P_{m+1})}^{-1} d_k(W^{-\frac1p}f)\|^2_{\mathbb{C}^{d}}\chi_{P_{m+1}}\\
		&\leq C_\gamma^2\sum_{P_{m+1}\in\mathcal P_{m+1}} \|W^{\frac{1}{p}} \widehat{W}_{\kappa_1(P_{m+1})}\|^2 \Big(\mathbb{E}_{\kappa_1(P_{m+1})}\|\widehat{W}_{\kappa_1(P_{m+1})}^{-1} W^{-\frac1p}f\|_{\mathbb{C}^{d}}\Big)^2\chi_{P_{m+1}}\\
		&=  C_\gamma^2\sum_{P_{m}\in\mathcal P_{m}}  \|W^{\frac{1}{p}} \widehat{W}_{\kappa_2(P_{m})}\|^2 \Big(\mathbb{E}_{\kappa_2(P_{m})}\|\widehat{W}_{\kappa_2(P_{m})}^{-1} W^{-\frac1p}f\|_{\mathbb{C}^{d}}\Big)^2\chi_{\{\tau_{P_{m}}<\infty\}}.
\end{split}
\end{equation}
On the other hand, it follows from Property \eqref{Pro-3} that
\begin{align*}
&\sum_{P_{m+1}\in\mathcal P_{m+1}} \|W^{\frac1p}d_{\kappa_2(P_{m+1})}(W^{-\frac1p}f)\|_{\mathbb C^d}^2\chi_{P_{m+1}}\\
&\lesssim \sum_{P_{m+1}\in\mathcal P_{m+1}} \|W^{\frac1p}\mathbb E_{\kappa_2(P_{m+1})}(W^{-\frac1p}f)\|_{\mathbb C^d}^2\chi_{P_{m+1}}\\\
&\qquad +\sum_{P_{m+1}\in\mathcal P_{m+1}} \|W^{\frac1p}\mathbb E_{\kappa_2(P_{m+1})-1}(W^{-\frac1p}f)\|_{\mathbb C^d}^2\chi_{P_{m+1}}\\		
&\leq \sum_{P_{m+1}\in\mathcal P_{m+1}}  \|W^{\frac{1}{p}} \widehat{W}_{\kappa_2(P_{m+1})}\|^2 \Big(\mathbb{E}_{\kappa_2(P_{m+1})}\|\widehat{W}_{\kappa_2(P_{m+1})}^{-1} W^{-\frac1p}f\|_{\mathbb{C}^{d}}\Big)^2\chi_{P_{m+1}}\\
&\qquad+  C_\gamma^2\sum_{P_{m+1}\in\mathcal P_{m+1}}  \|W^{\frac{1}{p}} \widehat{W}_{\kappa_1(P_{m+1})}\|^2 \Big(\mathbb{E}_{\kappa_1(P_{m+1})}\|\widehat{W}_{\kappa_1(P_{m+1})}^{-1} W^{-\frac1p}f \|_{\mathbb{C}^{d}}\Big)^2\chi_{P_{m+1}}.
	\end{align*}
This, together with  $\kappa_1(P_{m+1})=\kappa_2(P_m)$ and  \eqref{observation-a-2}, yields
\begin{equation}\label{estimate-b2-2}
\begin{split}
&\sum_{P_{m+1}\in\mathcal P_{m+1}} \|W^{\frac1p}d_{\kappa_2(P_{m+1})}(W^{-\frac1p}f)\|_{\mathbb C^d}^2\chi_{P_{m+1}}\\
&\leq \sum_{P_{m+1}\in\mathcal P_{m+1}} \|W^{\frac{1}{p}} \widehat{W}_{\kappa_2(P_{m+1})}\|^2 \Big(\mathbb{E}_{\kappa_2(P_{m+1})}\|\widehat{W}_{\kappa_2(P_{m+1})}^{-1} W^{-\frac1p}f\|_{\mathbb{C}^{d}}\Big)^2\chi_{P_{m+1}}\\
&\qquad +  C_\gamma^2\sum_{P_{m}\in\mathcal P_{m}} \|W^{\frac{1}{p}} \widehat{W}_{\kappa_2(P_{m})}\|^2 \Big(\mathbb{E}_{\kappa_2(P_{m})}\|\widehat{W}_{\kappa_2(P_{m})}^{-1} W^{-\frac1p}f\|_{\mathbb{C}^{d}}\Big)^2\chi_{\{\tau_{P_{m}}<\infty\}}.
\end{split}
\end{equation}
Plugging the estimates \eqref{estimate-b2-1}, \eqref{estimate-b2-2}  into \eqref{bm2}, we conclude that
\begin{align*}
b_{m,2}^2& \leq 2C_\gamma^2 \sum_{P_{m} \in \mathcal{P}_{m}}\|W^{\frac{1}{p}} \widehat{W}_{\kappa_{2}(P_{m})}\|^2 \Big(\mathbb{E}_{\kappa_{2}\left(P_{m}\right)}\|\widehat{W}_{\kappa_{2}(P_{m})}^{-1} W^{-\frac1p}f\|_{\mathbb{C}^{d}}\Big)^2\chi_{\{\tau_{P_{m}}<\infty\}}\\
		&\qquad + \sum_{P_{m+1} \in \mathcal{P}_{m+1}}\|W^{\frac{1}{p}} \widehat{W}_{\kappa_{2}(P_{m+1})}\|^2 \Big(\mathbb{E}_{\kappa_{2}(P_{m+1})}\|\widehat{W}_{\kappa_{2}(P_{m+1})}^{-1} W^{-\frac1p}f\|_{\mathbb{C}^{d}}\Big)^2\chi_{P_{m+1}}+b_{m+1}^2.
\end{align*}
A combination of the  estimates for $b_{m,1}$ and  $b_{m,2}$ leads to the desired inequality \eqref{claimsq}.

We now turn to verify
\eqref{after-iteration}. Given $N>1$, applying  \eqref{claimsq} for $N-1$ times, we get
\begin{equation*}
\begin{split}
b_1^2&\leq b_{N}^2+2C_\gamma^2 \sum_{m=1}^{N-1}\sum_{P_{m} \in \mathcal{P}_{m}}\|W^{\frac{1}{p}} \widehat{W}_{\kappa_{2}(P_{m})}\|^2 \Big(\mathbb{E}_{\kappa_{2}(P_{m})}\|\widehat{W}_{\kappa_{2}(P_{m})}^{-1} W^{-\frac1p}f\|_{\mathbb{C}^{d}}\Big)^2\chi_{P_m}\\
&\qquad + \sum_{m=2}^N \sum_{P_{m} \in \mathcal{P}_{m}}\|W^{\frac{1}{p}} \widehat{W}_{\kappa_{1}(P_{m})}\|^2 \Big(\mathbb{E}_{\kappa_{1}(P_{m})}\|\widehat{W}_{\kappa_{1}(P_{m})}^{-1} W^{-\frac1p}f\|_{\mathbb{C}^{d}}\Big)^2\chi_{P_{m}}
\end{split}
\end{equation*}
By the construction of the principal set, for any $P\in\mathcal P_n$, $n\geq 2$, there exists $P'$ such that  $P'\in\mathcal P_{n-1}$ and $\kappa_2(P')=\kappa_1(P)$. Moreover, for any fixed $P'\in\mathcal P_{n-1}$, if $P_1,P_2\in \mathcal P_n$ and $P_1,P_2\subset P'$, then $P_1=P_2$ or $P_1\cap P_2=\emptyset$. Thus we have
\begin{equation*}
\begin{split}
b_1^2&\leq b_{N}^2+(2C_\gamma^2+1) \sum_{m=1}^{N-1}\sum_{P_{m} \in \mathcal{P}_{m}}\|W^{\frac{1}{p}} \widehat{W}_{\kappa_{2}(P_{m})}\|^2 \Big(\mathbb{E}_{\kappa_{2}(P_{m})}\|\widehat{W}_{\kappa_{2}(P_{m})}^{-1} W^{-\frac1p}f\|_{\mathbb{C}^{d}}\Big)^2\chi_{P_m}.
\end{split}
\end{equation*}
The proof is finished.
\end{proof}

Besides the iterative formula of $(b_m)_{m\in\mathbb N}$, we need two more lemmas.

\begin{lemma}\label{SWf=0}
We have
\begin{equation*}
S_W f \chi_{\left\{\tau_{1}=\infty\right\}}=0.
\end{equation*}
\end{lemma}
	
\begin{proof}
Recall that $n=0$. We have, on the set $\{\tau_{1}=\infty\}$, that
$$0=\gamma_1(0,\infty)=\frac{\left(\sum_{ i=1}^{  \infty}\|\widehat W_{0}^{-1}d_i( W^{-1/p}f)\|^2_{\mathbb C^d}\right)^{1/2}}{\mathbb E_{0}\|\widehat W_0^{-1}W^{-1/p}f\|_{\mathbb C^d}}\chi_{\{\mathbb E_{0}\|\widehat W_0^{-1}W^{-1/p}f\|_{\mathbb C^d}>0\}}.$$
Since   $\widehat W_0^{-1}$ is positive definite, it follows that $d_i(W^{-1/p}f)=0$ for all $i\geq 1$. The assertion is verified.
\end{proof}

\begin{lemma}\label{swf=01}
Let $P_1\in\mathcal P_1$. We have
\begin{equation*}\sum_{i<\kappa_2(P_1)}\|W^{\frac1p}d_i(W^{-\frac1p}f)\|_{\mathbb C^d}^2\chi_{P_1}=0.
\end{equation*}
\end{lemma}

\begin{proof}
Given $P_1\in\mathcal P_1$, observe that on the set $P_1$ we have
$$0=\gamma_1(0,\kappa_2(P_1)-1)=\frac{\left(\sum_{ i=1}^{\kappa_2(P_1)-1}\|\widehat W_{0}^{-1}d_i( W^{-1/p}f)\|^2_{\mathbb C^d}\right)^{1/2}}{\mathbb E_{0}\|\widehat W_0^{-1}W^{-1/p}f\|_{\mathbb C^d}}\chi_{\{\mathbb E_{0}\|\widehat W_0^{-1}W^{-1/p}f\|_{\mathbb C^d}>0\}}.$$   Since  $\widehat W_0^{-1}$ is positive definite, it follows that $d_i(W^{-1/p}f)=0$ for $1\leq i<\kappa_2(P_1)$. This gives the desired result.
\end{proof}

We now are full equipped to prove Proposition \ref{Sdominition}.

\begin{proof}[Proof of Proposition \ref{Sdominition}]
By Lemma \ref{SWf=0} and Lemma \ref{swf=01}, we have
	\begin{align*}
		\quad S_W^2 f&=  S_W^2 f\chi_{\left\{\tau_{1}<\infty\right\}} = \sum_{P_1\in\mathcal P_1}(S_W^2 f)\chi_{P_1}\\
		&= \sum_{P_1\in\mathcal P_1}\|W^{\frac{1}{p}}d_{\kappa_2(P_1)}(W^{-\frac1p}f)\|_{\mathbb{C}^{d}}^2\chi_{P_1}+\sum_{P_1\in\mathcal P_1}\sum_{i> \kappa_2(P_1)}\|W^{\frac{1}{p}}d_i(W^{-\frac1p}f)\|_{\mathbb{C}^{d}}^2\chi_{P_1}\\ &=\sum_{P_1\in \mathcal P_1} \|W^{\frac1p}d_{\kappa_2(P_1)}(W^{-\frac1p}f)\|_{\mathbb C^d}^2\chi_{P_1}+b_1^2\\
	&\lesssim  \sum_{P_1\in \mathcal P_1} \|W^{\frac1p}\mathbb E_{\kappa_2(P_1)}(W^{-\frac1p}f)\|_{\mathbb C^d}^2\chi_{P_1}+\sum_{P_1\in \mathcal P_1} \|W^{\frac1p}\mathbb E_{\kappa_2(P_1)-1}(W^{-\frac1p}f)\|_{\mathbb C^d}^2\chi_{P_1}+b_1^2.
	\end{align*}
Observe that
$$\|W^{\frac1p}\mathbb E_{\kappa_2(P_1)}(W^{-\frac1p}f)\|_{\mathbb C^d}\leq\|W^{\frac1p}\widehat W_{\kappa_2(P_1)}\|\Big(\mathbb E_{\kappa_2(P_1)}\|\widehat W_{\kappa_2(P_1)}^{-1}W^{-\frac1p}f\|_{\mathbb C^d}\Big).$$
Moreover, according to the definition of $\kappa_2(P_1)$, we  have
\begin{align*}
\|W^{\frac1p}\mathbb E_{\kappa_2(P_1)-1}(W^{-\frac1p}f)\|_{\mathbb C^d}\leq\|W^{\frac1p}\widehat W_{\kappa_1(P_1)}\| \cdot \|\widehat W_{\kappa_1(P_1)}^{-1}\mathbb E_{\kappa_2(P_1)-1}(W^{-\frac1p}f\|_{\mathbb C^d})\|=0.
\end{align*}
Then, it follows that
\begin{align*}
S_W^2f\leq b_1^2& +\sum_{P_1\in \mathcal P_1} \|W^{\frac1p}\widehat W_{\kappa_2(P_1)}\|^2\Big(\mathbb E_{\kappa_2(P_1)}\|\widehat W_{\kappa_2(P_1)}^{-1}W^{-\frac1p}f\|_{\mathbb C^d}\Big)^2\chi_{P_1}.
\end{align*}
Given an integer $N>1$, we deduce from \eqref{after-iteration} that
	\begin{align*}
		S_W^2f\leq b_N^2+2(C_\gamma^2+1)\sum_{i=1}^{N-1}\sum_{P_i\in \mathcal P_i} \|W^{\frac1p}\widehat W_{\kappa_2(P_i)}\|^2\Big(\mathbb E_{\kappa_2(P_i)}\|\widehat W_{\kappa_2(P_i)}^{-1}W^{-\frac1p}f\|_{\mathbb C^d}\Big)^2\chi_{P_i}.
	\end{align*}
Note that the function $b_N$ is supported by $\bigcup_{P\in \mathcal P_N}P$. We conclude the result by sending $N\to \infty$ and Property \eqref{Pro-6}.
\end{proof}

\begin{proof}[Proof of Theorem \ref{main2}]
A combination of Proposition \ref{Sdominition} with  Proposition \ref{5.4}  immediately yields the desired assertion.
\end{proof}

\medskip

\section{Proof of Theorem \ref{main1}}
This section devotes to the sharp  weighted norm estimate for  martingale square functions in the scalar setting. With some proper modification on the previous arguments, we show that in the scalar setting the exponent of the weight in Theorem \ref{main2} can be improved to $\max\{1/2 ,1/(p-1)\}$ which is the best possible. Surprisingly, it seems that this result previously has been only known for dyadic square functions (\cite{MR2854179,MR2697378,MR1771755,MR1873024,MR1748283}), continuous-time martingales with continuous
path (\cite[Theorem 4.1]{MR3748572}), and  discrete-time martingales associated with atomic filtration (\cite[Theorem 3.6]{MR3985127}).

Suppose that $w$ is a classical weight on $(\Omega,\mathcal F,\mathbb P)$. Now the reducing operators $\widetilde w_n$, $\widehat w_n$ can be taken as
$$\widetilde w_n=(\mathbb E_n w)^{\frac1p},\quad \widehat w_n=\big(\mathbb E_n(w^{-\frac {p'}p})\big)^{\frac1{p'}}.$$
For $1<p<\infty$, we say that $w$ satisfies the $A_p$ condition (denoted by $w\in A_p$) if
$[w]_{A_p}:=\sup_{n}\|\widetilde w_n  \widehat w_n\|_{\infty}^p$ is finite. Denote by  $\mathbb E_n^w$ the conditional expectation, with respect to $\mathcal F_n$ and the measure induced by $w$. Then we have the following simple fact:
\begin{equation}\label{simple}
\mathbb E_n^w(f)=\mathbb E_n(wf)(\mathbb E_n(w))^{-1},\quad \forall f\in L_1^w(\Omega).
\end{equation}
For $r\geq 1$, we modify the sparse operators  as follows:
$$\mathcal T_{w,r}f=\Big(\sum_{P \in \mathcal{P}}w^{\frac{r}{p}} \big(\mathbb{E}_{\kappa_{2}(P)}(|w^{-\frac1p}f| )\big)^r\chi_{P}\Big)^{\frac1r}.$$

\begin{lemma}\label{sparse-enhanced}
	Let $1<r<p<\infty$. Given a scalar weight $w\in A_p$, we have
	\begin{equation*}\label{Lem-main}
		\left\|\mathcal T_{w,r}f\right\|_{L_p(\Omega)} \lesssim_{p,d}[w]_{A_p}^{\max\{\frac1r ,\frac1{p-1}\}} \|f\|_{L_p(\Omega)},\quad\forall f\in L_p(\Omega).
	\end{equation*}
\end{lemma}

\begin{proof}
Let $s=p/r$. By duality and Property \eqref{Pro-2}, we have
\begin{align*}
\left\| \mathcal T_{w,r}f\right\|_{L_{p}(\Omega)}^r&= \sup_{\substack{ g \footnotesize{\mbox{ is positive}}\\\|g\|_{L_{s'}\leq 1}}}\int_{\Omega}\sum_{P \in \mathcal{P}}w^{\frac{r}{p}} \big(\mathbb{E}_{\kappa_{2}(P)}(|w^{-\frac1p}f| )\big)^r\chi_{P} gd\mathbb P\\
&=\sup_{\substack{ g \footnotesize{\mbox{ is positive}}\\\|g\|_{L_{s'}}\leq 1}}\sum_{P \in \mathcal{P}}\int_{\Omega}\mathbb E_{\kappa_2(P)}(w^{\frac{r}{p}}g) \big(\mathbb{E}_{\kappa_{2}(P)}(|w^{-\frac1p}f|)\big)^r\chi_{P} d\mathbb P\\
&\leq \sup_{\substack{g \footnotesize{\mbox{ is positive}}\\\|g\|_{L_{s'}}\leq 1}}\Big[\sum_{P\in\mathcal P_{1}}\int_{P} \big(\mathbb{E}_{\kappa_{2}(P)}(|w^{-\frac1p}f|)\big)^r\mathbb E_{\kappa_2(P)}(w^{\frac{r}{p}} g)d\mathbb P\\
&\qquad +\sum_{P\in\bigcup_{i=2}^\infty\mathcal P_{i}}\int_{E(P)} \big(\mathbb{E}_{\kappa_{2}(P)}(|w^{-\frac1p}f|)\big)^r\mathbb E_{\kappa_2(P)}(w^{\frac{r}{p}} g)d\mathbb P\Big],
	\end{align*}
where the last estimate is due to Property \eqref{Pro-5}, $\chi_P\leq 2\mathbb E_{\kappa_2(P)}(\chi_{E(P)})$ for  every $P\in\bigcup_{n=2}^\infty\mathcal P_n$.
Then by H\"older's inequality, we find
\begin{equation}\label{estimate-1-2-3}
\left\| \mathcal T_{w,r}f\right\|_{L_{p}(\Omega)}^r\leq \sup_{\substack{g \footnotesize{\mbox{ is positive}}\\\|g\|_{L_{s'}}\leq 1}}{\rm I}\cdot {\rm II} \cdot {\rm III},
\end{equation}
where
\begin{align*}
{\rm I}&:= \Big(\sum_{P\in\mathcal P_1}\Big\| \big(\mathbb{E}_{\kappa_{2}(P)}(|w^{-\frac1p}f|)\big)^r
\big(\mathbb E_{\kappa_{2}(P)}(w^{-\frac{p'}{p}}\chi_{P} )\big)^{\frac{1}{s}}\big(\mathbb{E}_{\kappa_{2}(P)}(w^{-\frac{p'}{p}})\big)^{-r}\Big\|_{L_s(\Omega)}^s
\\
&\quad + \sum_{P\in\bigcup_{i=2}^\infty\mathcal P_{i}}
\Big\| \big(\mathbb{E}_{\kappa_{2}(P)}(|w^{-\frac1p}f|)\big)^r
\big(\mathbb E_{\kappa_{2}(P)}(w^{-\frac{p'}{p}}\chi_{E(P)} )\big)^{\frac{1}{s}}\big(\mathbb{E}_{\kappa_{2}(P)}(w^{-\frac{p'}{p}})\big)^{-r}\Big\|_{L_s(\Omega)}^s
\Big)^{1/s},\\
{\rm II}&:= \Big(\sum_{P\in\mathcal P_1}\Big\| \mathbb{E}_{\kappa_{2}(P)}(w^{\frac{r}{p}}g)
\big(\mathbb E_{\kappa_{2}(P)}(w\chi_{P} )\big)^{\frac{1}{s'}}\big(\mathbb{E}_{\kappa_{2}(P)}(w)\big)^{-1}\Big\|_{L_{s'}(\Omega)}^{s'}\\
&\quad+ \sum_{P\in\bigcup_{i=2}^\infty\mathcal P_{i}}\Big\| \mathbb{E}_{\kappa_{2}(P)}(w^{\frac{r}{p}}g)
\big(\mathbb E_{\kappa_{2}(P)}(w\chi_{E(P)} )\big)^{\frac{1}{s'}}\big(\mathbb{E}_{\kappa_{2}(P)}(w)\big)^{-1}\Big\|_{L_{s'}(\Omega)}^{s'}\Big)^{1/s'},
\\
{\rm III} &:= \max\bigg{\{}\sup_{P\in\mathcal P_1}\big\|\big(\mathbb E_{\kappa_{2}(P)}(w^{-\frac{p'}{p}}\chi_{P} )\big)^{-\frac{1}{s}}\big(\mathbb E_{\kappa_{2}(P)}(w^{-\frac{p'}{p}})\big)^{r}\\
&\qquad\qquad\qquad\qquad\qquad\qquad\cdot\big(\mathbb E_{\kappa_{2}(P)}(w\chi_{P} )\big)^{-\frac{1}{s'}}\big(\mathbb{E}_{\kappa_{2}(P)}(w)\big)\chi_P\big\|_{L_\infty(\Omega)},\\
&\quad\quad\quad\sup_{P\in\bigcup_{i=2}^\infty\mathcal P_{i}} \big\|\big(\mathbb E_{\kappa_{2}(P)}(w^{-\frac{p'}{p}}\chi_{E(P)} )\big)^{-\frac{1}{s}}\big(\mathbb E_{\kappa_{2}(P)}(w^{-\frac{p'}{p}})\big)^{r}	\\&\qquad\qquad\qquad\qquad\qquad\qquad\cdot\big(\mathbb E_{\kappa_{2}(P)}(w\chi_{E(P)} )\big)^{-\frac{1}{s'}}\big(\mathbb{E}_{\kappa_{2}(P)}(w)\big)\chi_{E(P)}\big\|_{L_\infty(\Omega)}\bigg{\}}
\end{align*}
Let us first deal with $\rm I$. Note that
\begin{align*}
{\rm I}^s &\leq \sum_{P\in\mathcal P_1}\int_P \big(\mathbb{E}_{\kappa_{2}(P)}(|w^{-\frac1p}f|)\big)^{p} \big(\mathbb{E}_{\kappa_{2}(P)}(w^{-\frac{p'}{p}})\big)^{-p}w^{-\frac{p'}{p}}d\mathbb P\\
&\quad +\sum_{P\in\bigcup_{i=2}^\infty\mathcal P_{i}}\int_{E(P)} \big(\mathbb{E}_{\kappa_{2}(P)}(|w^{-\frac1p}f|)\big)^{p} \big(\mathbb{E}_{\kappa_{2}(P)}(w^{-\frac{p'}{p}})\big)^{-p}w^{-\frac{p'}{p}}d\mathbb P
\end{align*}
Denote by $v$ the dual weight of $w$ (i.e., $v=w^{-\frac{p'}{p}}$) and set $w^{-1/p}f=v\widetilde f$. Then by \eqref{simple} and the unweighted Doob inequality, we get
\begin{align*}
{\rm I}^s &\leq \sum_{P\in\mathcal P_1}\int_P \big(\mathbb{E}_{\kappa_{2}(P)}(|v\widetilde f|)\big)^{p} \big(\mathbb{E}_{\kappa_{2}(P)}(v)\big)^{-p}vd\mathbb P \\ &\quad +\sum_{P\in\bigcup_{i=2}^\infty\mathcal P_{i}}\int_{E(P)} \big(\mathbb{E}_{\kappa_{2}(P)}(|v\widetilde f|)\big)^{p} \big(\mathbb{E}_{\kappa_{2}(P)}(v)\big)^{-p}vd\mathbb P\\
&\leq \sum_{P\in\mathcal P_1}\int_P \big(\sup_{n} |\mathbb E_n^v(\widetilde f)|\big)^p vd\mathbb P+\sum_{P\in\bigcup_{i=2}^\infty\mathcal P_{i}}\int_{E(P)}\big(\sup_{n} |\mathbb E_n^v(\widetilde f)|\big)^p vd\mathbb P\\
&\lesssim \int_{\Omega} \big(\sup_{n} |\mathbb E_n^v(\widetilde f)|\big)^p vd\mathbb P \lesssim_p \int_{\Omega} |\widetilde f|^p vd\mathbb P=\|f\|_{L_p(\Omega)}^p.
\end{align*}
Similarly, set $w\widetilde g=w^{r/p}g$. Using again \eqref{simple} and the unweighted Doob inequality, we obtain
\begin{align*}
{\rm II}^{s'} & \leq \sum_{P\in\mathcal P_1}\int_P \big(\mathbb{E}_{\kappa_{2}(P)}(w^{\frac{r}{p}}g)\big)^{s'}\big(\mathbb{E}_{\kappa_{2}(P)}(w)\big)^{-s'}wd\mathbb P\\
&\quad +\sum_{P\in\bigcup_{i=2}^\infty\mathcal P_{i}}\int_{E(P)} \big(\mathbb{E}_{\kappa_{2}(P)}(w^{\frac{r}{p}}g)\big)^{s'}\big(\mathbb{E}_{\kappa_{2}(P)}(w)\big)^{-s'}wd\mathbb P\\
&\lesssim \int_\Omega \big(\sup_{n} \mathbb E_n^w(\widetilde g)\big)^{s'} w d\mathbb P \lesssim_s \int_{\Omega} \widetilde g^{s'} wd\mathbb P =\|g\|_{L_{s'}(\Omega)}^{s'}.
\end{align*}
It remains to handle $\rm III$. For convenience, let us denote
\begin{align*}
{\rm III}_1&:=\sup_{P\in\mathcal P_1}\big\|\big(\mathbb E_{\kappa_{2}(P)}(w^{-\frac{p'}{p}}\chi_{P} )\big)^{-\frac{1}{s}}\big(\mathbb E_{\kappa_{2}(P)}(w^{-\frac{p'}{p}})\big)^{r}\\
&\qquad\qquad\qquad\qquad\qquad\qquad\cdot\big(\mathbb E_{\kappa_{2}(P)}(w\chi_{P} )\big)^{-\frac{1}{s'}}\big(\mathbb{E}_{\kappa_{2}(P)}(w)\big)\chi_P\big\|_{L_\infty(\Omega)},\\
{\rm III}_2&:=\sup_{P\in\bigcup_{i=2}^\infty\mathcal P_{i}} \big\|\big(\mathbb E_{\kappa_{2}(P)}(w^{-\frac{p'}{p}}\chi_{E(P)} )\big)^{-\frac{1}{s}}\big(\mathbb E_{\kappa_{2}(P)}(w^{-\frac{p'}{p}})\big)^{r}	\\&\qquad\qquad\qquad\qquad\qquad\qquad\cdot\big(\mathbb E_{\kappa_{2}(P)}(w\chi_{E(P)} )\big)^{-\frac{1}{s'}}\big(\mathbb{E}_{\kappa_{2}(P)}(w)\big)\chi_{E(P)}\big\|_{L_\infty(\Omega)}.
\end{align*}
By the $A_p$ condition,  we have
\begin{align*}
{\rm III}_2 &=\sup_{P\in\bigcup_{i=2}^\infty\mathcal P_{i}} \big\|\big(\mathbb E_{\kappa_{2}(P)}(w^{-\frac{p'}{p}}\chi_{E(P)} )\big)^{-\frac{1}{s}}\big(\mathbb E_{\kappa_{2}(P)}(w^{-\frac{p'}{p}})\big)^{r}	\\&\qquad\qquad\qquad\qquad\qquad\qquad\cdot\big(\mathbb E_{\kappa_{2}(P)}(w\chi_{E(P)} )\big)^{-\frac{1}{s'}}\big(\mathbb{E}_{\kappa_{2}(P)}(w)\big)\chi_{E(P)}\big\|_{L_\infty(\Omega)}\\
&\leq \sup_{P\in\bigcup_{i=2}^\infty\mathcal P_{i}} \|\big[\big(\mathbb E_{\kappa_{2}(P)}(w^{-\frac{p'}{p}})\big)^{p-1} \mathbb{E}_{\kappa_{2}(P)}(w)\big]^{\frac1s}\|_{L_\infty(\Omega)}\\
&\quad\cdot \big\|\big[\big(\mathbb E_{\kappa_{2}(P)}(w\chi_{E(P)} )\big)^{-1}\mathbb{E}_{\kappa_{2}(P)}(w)\big]^{\frac{1}{s'}}\\
&\qquad \cdot \big[\big(\mathbb E_{\kappa_{2}(P)}(w^{-\frac{p'}{p}}\chi_{E(P)} )\big)^{-1}\big(\mathbb E_{\kappa_{2}(P)}(w^{-\frac{p'}{p}})\big)\big]^{\frac1s}\chi_{E(P)}\big\|_{L_\infty(\Omega)}\\
&\leq [w]_{A_p}^{\frac1s}\sup_{P\in\bigcup_{i=2}^\infty\mathcal P_{i}}
\big\|\big[\big(\mathbb E_{\kappa_{2}(P)}(w\chi_{E(P)} )\big)^{-1}\mathbb{E}_{\kappa_{2}(P)}(w)\big]^{\frac{1}{s'}}\\
&\qquad \cdot \big[\big(\mathbb E_{\kappa_{2}(P)}(w^{-\frac{p'}{p}}\chi_{E(P)} )\big)^{1-p}\big(\mathbb E_{\kappa_{2}(P)}(w^{-\frac{p'}{p}})\big)^{p-1}\big]^{\frac1{s(p-1)}}\chi_{E(P)}\big\|_{L_\infty(\Omega)}.
\end{align*}
Note that
$$\big(\mathbb E_{\kappa_{2}(P)}(w\chi_{E(P)} )\big)^{-1}\mathbb{E}_{\kappa_{2}(P)}(w)\geq 1$$
and
$$\big(\mathbb E_{\kappa_{2}(P)}(w^{-\frac{p'}{p}}\chi_{E(P)} )\big)^{1-p}\big(\mathbb E_{\kappa_{2}(P)}(w^{-\frac{p'}{p}})\big)^{p-1}\geq 1.$$
Therefore,
\begin{equation}\label{III3}
\begin{split}
{\rm III}_2&\leq [w]_{A_p}^{\frac1s} \sup_{P\in\bigcup_{i=2}^\infty\mathcal P_{i}}\big\|\big[\big(\mathbb E_{\kappa_{2}(P)}(w\chi_{E(P)} )\big)^{-1}\mathbb{E}_{\kappa_{2}(P)}(w)\\
&\qquad \cdot\big(\mathbb E_{\kappa_{2}(P)}(w^{-\frac{p'}{p}}\chi_{E(P)} )\big)^{1-p}\big(\mathbb E_{\kappa_{2}(P)}(w^{-\frac{p'}{p}})\big)^{p-1}\big]^{^{\max\{\frac r{p(p-1)},\frac 1{s'}\}}}\chi_{E(P)}\big\|_{L_\infty(\Omega)}.
\end{split}
\end{equation}
For every $P\in\bigcup_{i=2}^\infty\mathcal P_{i}$, by Property \eqref{Pro-5} and H\"older's inequality, we have
	\begin{align*}
\chi_{E(P)}\leq \chi_P\lesssim \mathbb E_{{{\kappa_{2}\left(P\right)}}}(\chi_{E(P)})\leq \big(\mathbb E_{{{\kappa_{2}\left(P\right)}}}(w\chi_{E(P)})\big)^{\frac 1p} \big(\mathbb E_{{{\kappa_{2}\left(P\right)}}}(w
		^{-\frac {p'}p}\chi_{E(P)})\big)^{\frac 1{p'}},
		\end{align*}
which implies
$$\chi_{E(P)}\lesssim \mathbb E_{{{\kappa_{2}(P)}}}(w\chi_{E(P)}) \big(\mathbb E_{{{\kappa_{2}(P)}}}(w
		^{-\frac {p'}p}\chi_{E(P)})\big)^{p-1}.$$
Hence,
\begin{equation*}\label{useful-fact}
\begin{split}
&\big(\mathbb E_{{{\kappa_{2}(P)}}}(w^{-\frac {p'}p}\chi_{E(P)})\big)^{1-p}\big(\mathbb E_{\kappa_{2}(P)}(w^{-\frac{p'}{p}})\big)^{p-1}	\big(\mathbb E_{{{\kappa_{2}(P)}}}(w\chi_{E(P)}) \big)^{-1}\mathbb E_{\kappa_{2}(P)}(w)\chi_{E(P)}\\
&\quad \lesssim \big(\mathbb E_{\kappa_{2}(P)}(w^{-\frac{p'}{p}})\big)^{p-1}\mathbb E_{\kappa_{2}(P)}(w)\leq [w]_{A_p}.
\end{split}
\end{equation*}
Plugging this estimate into \eqref{III3}, we get
\begin{align*}
 {\rm III}_2&\lesssim [w]_{A_p}^{\frac1s} [w]_{A_p}^{\max\{\frac r{p(p-1)},\frac1{s'}\}}=[w]_{A_p}^{\max\{\frac{r}{p-1},1\}}.
\end{align*}
Similarly, one can show that
$${\rm III}_1\leq [w]_{A_p}^{\max\{\frac{r}{p-1},1\}}.$$
Combining the above estimates of $\rm I, II, III$ with \eqref{estimate-1-2-3}, we conclude that
	\begin{align*}
		\left\| \mathcal T_{w,r}f\right\|_{L_{p}(\Omega)}
				&\lesssim [w]_{A_p}^{\max\{\frac1 r,\frac1{p-1}\}}\|f\|_{L_p(\Omega)}.
	\end{align*}
The proof is complete.
\end{proof}

It is easy to see that Lemma \ref{5.4'} still holds for $\mathcal T_{w,p}$. As a consequence of this result and Lemma \ref{sparse-enhanced}, the assertion in  Proposition \ref{5.4} can be enhanced  as follows.

\begin{corollary}\label{key-cor}
 Let $1<p<\infty$. Given a  scalar
	 weight $w\in A_p$, we have
	\begin{equation}\label{Lem-main-2}
		\left\|\mathcal T_{w,2}f\right\|_{L_p(\Omega)} \lesssim_{p,d}[w]_{A_p}^{\max\{\frac12 ,\frac1{p-1}\}} \|f\|_{L_p(\Omega)},\quad\forall f\in L_p(\Omega).
	\end{equation}
\end{corollary}
\begin{proof} Clearly, for $2<p<\infty$, \eqref{Lem-main-2} is just a special case of Lemma \ref{sparse-enhanced}. As for $1<p\leq 2$,   we have
\begin{align*}
\mathcal T_{w,2} f \leq  \mathcal T_{w,p}f.
\end{align*}
The estimate  \eqref{Lem-main-2} now follows from  Lemma \ref{5.4'}.
\end{proof}

\begin{proof}[Proof of Theorem \ref{main1}]
Note that Proposition \ref{Sdominition} remains valid for $S_w$ and $\mathcal T_{w,2}$. That is,
$$S_w f \lesssim_{p} \mathcal T_{w,2} f,\quad\forall f\in L_p(\Omega).$$
This domination together with Corollary \ref{key-cor} gives
$$\|S_wf\|_{L_p(\Omega)}\lesssim_{p}[w]_{A_p}^{\max\{\frac12,\frac{1}{p-1}\}}\|f\|_{L_p(\Omega)},\quad \forall f\in L_p(\Omega).$$
Replacing $f$ by $w^{1/p}f$, we obtain the desired result.
\end{proof}

\bibliographystyle{alpha,amsplain}	
	
\end{document}